\documentclass[sn-mathphys,Numbered]{sn-jnl}

\usepackage{geometry}
\usepackage{graphicx}
\usepackage{float} 
\usepackage{subfigure}
\usepackage{multirow}
\usepackage{amsmath}
\usepackage{amssymb}
\usepackage{amsfonts}
\usepackage{amsthm}
\usepackage{mathrsfs}
\usepackage[title]{appendix}
\usepackage{xcolor}
\usepackage{textcomp}
\usepackage{manyfoot}
\usepackage{booktabs}
\usepackage{algorithm}
\usepackage{algorithmicx}
\usepackage{algpseudocode}
\usepackage{array}
\usepackage{listings}
\usepackage{hyperref}
\usepackage{microtype} 
\usepackage{bm}
\usepackage{booktabs}
\usepackage{multirow}

\theoremstyle{thmstyleone}
\newtheorem{theorem}{Theorem}[section]
\newtheorem{assumption}{Assumption}[section]
\newtheorem{lemma}{Lemma}[section]
\newtheorem{proposition}[theorem]{Proposition}
\newtheorem{corollary}{Corollary}[section]
\newtheorem{example}{Example}[section]

\theoremstyle{thmstyletwo}
\newtheorem{remark}{Remark}[section]

\theoremstyle{thmstylethree}
\newtheorem{definition}{Definition}[section]

\begin{document}

\title[Tikhonov-Regularized Second-Order Primal-Dual Dynamics for Convex-Concave Bilinear Saddle Point Problems]{Tikhonov-Regularized Second-Order Primal-Dual Dynamics for Convex-Concave Bilinear Saddle Point Problems}

\author[1]{\fnm{Bohan} \sur{Zhang}}\email{bohanzhang@std.uestc.edu.cn}

\author*[1]{\fnm{Xiaojun} \sur{Zhang}}\email{sczhxj@uestc.edu.cn}

\affil[1]{University of Electronic Science and Technology of China, Chengdu, Sichuan 611731, People's
	Republic of China}

\abstract{
	This paper develops a Lyapunov-based coefficient-compatible framework for a class of second-order primal-dual dynamical systems with Tikhonov regularization. Unlike existing accelerated inertial primal-dual dynamics, whose analyses are typically carried out under prescribed damping profiles, the proposed framework characterizes the interplay among the viscous damping coefficient, the time-scaling coefficient, the linear extrapolation parameter, and the Tikhonov regularization parameter through differential compatibility conditions, thereby accommodating time-varying coefficients beyond prescribed damping profiles. Within this framework, we investigate two cases in which the Tikhonov regularization parameter decays to zero at different rates. When the Tikhonov regularization parameter decays to zero relatively rapidly, the fast convergence rate of the primal-dual gap is preserved, whereas a slower decay rate yields refined asymptotic convergence estimates. We further show that a sufficiently slow decay ensures strong convergence of the trajectories to the minimum-norm saddle point. To establish the latter property, we introduce a moving-saddle Lyapunov framework centered at the instantaneous regularized saddle point and derive verifiable coefficient conditions ensuring the required tracking behavior. Finally, three numerical examples are presented to illustrate the theoretical results.
}


\maketitle

\section{Introduction}\label{sec1}
In recent years, the close connection between optimization algorithms and second-order dynamical systems has attracted considerable attention. On the one hand, continuous-time dynamical systems can be derived from classical optimization algorithms \cite{ref13,shoulian}. On the other hand, appropriate time discretizations of continuous-time dynamical systems may yield optimization algorithms that inherit the corresponding convergence properties \cite{ref8,ref9,ref10,ref11,ref12}. As a fundamental class of optimization problems, saddle point problems arise in a wide range of applications, including machine learning \cite{ref4,refnew2}, image reconstruction \cite{refnew4,ref6}, and compressed sensing \cite{refnew5,refnew6}. Consequently, the development and analysis of second-order dynamical systems for solving saddle point problems have become an active area of research.

Early studies of second-order dynamical systems for optimization primarily focused on the following unconstrained optimization problem:
\begin{equation}\label{eq1}
	\min_{x\in\mathcal{X}} f(x),
\end{equation}
where $\mathcal{X}$ is a real Hilbert space and $f:\mathcal{X}\to\mathbb{R}$ is a continuously differentiable convex function with a Lipschitz continuous gradient. To solve the problem \eqref{eq1}, Su et al. \cite{ref9} proposed the following second-order dynamical system:
\[
	\ddot{x}(t)+\frac{\alpha}{t}\dot{x}(t)+\nabla f(x(t))=0.\tag{$\text{AVD}_{\alpha}$}
\]
They showed that $\alpha=3$ is the critical threshold for achieving the convergence rate $\mathcal{O}\left(\frac{1}{t^2}\right)$. Here, $\frac{\alpha}{t}$ is an asymptotically vanishing viscous damping coefficient. Building on the dynamical system (AVD$_\alpha$), May \cite{ref16} proved that, for $\alpha>3$, the objective values converge at the rate $o\left(\frac{1}{t^2}\right)$. Attouch et al. \cite{ref17} and Apidopoulos et al. \cite{ref18} established that, for $\alpha\leq3$, the objective values along the trajectory satisfy the convergence estimate $\mathcal{O}\left(t^{-\frac{2\alpha}{3}}\right)$. Under additional assumptions on the objective function $f$, Aujol et al. \cite{ref19} further investigated the convergence behavior of the objective values for $\alpha>0$.

Early studies on the coupling of damped inertial dynamics with Tikhonov regularization focused on Polyak's heavy-ball system with friction \cite{ref15}, in which the damping coefficient $\alpha>0$ is constant. Attouch and Czarnecki \cite{new1} considered the following dynamical system:
\begin{equation*}
	\ddot{x}(t)+\alpha\dot{x}(t)+\nabla f(x(t))+\epsilon(t)x(t)=0.\tag{$\text{HBFC}$}
\end{equation*}
They proved that, under the condition $\int_{0}^{+\infty}\epsilon(t)\mathrm{d}t=+\infty$, every solution $x(t)$ of the dynamical system (HBFC) converges strongly to the minimum-norm element of $\operatorname{argmin} f$.

Building on (AVD$_{\alpha}$), Attouch et al. \cite{ref21} introduced a second-order dynamical system incorporating the Tikhonov regularization term $\epsilon(t)x(t)$. Their aim was to retain $\frac{\alpha}{t}$ as an asymptotically vanishing viscous damping coefficient while ensuring strong convergence of the generated trajectories. The resulting system is given by
\begin{equation*}
	\ddot{x}(t)+\frac{\alpha}{t}\dot{x}(t)+\nabla f(x(t))+\epsilon(t)x(t)=0,
	\tag{$\text{AVD}_{\alpha,\epsilon}$}
\end{equation*}
where $\alpha\geq3$ and $\epsilon\colon[t_0,+\infty)\to[0,+\infty)$. When $\epsilon(t)$ tends to zero sufficiently rapidly, (AVD$_{\alpha,\epsilon}$) retains the fast convergence properties of (AVD$_{\alpha}$). When $\epsilon(t)$ tends to zero sufficiently slowly, the trajectories generated by (AVD$_{\alpha,\epsilon}$) converge strongly to the minimum-norm minimizer of the problem \eqref{eq1}.

More generally, Riahi \cite{wuyueshu1} proposed the following generalized Tikhonov-regularized dynamical system involving a time-dependent viscous damping coefficient and a time-scaling coefficient:
\begin{equation*}
	\ddot{x}(t)+\alpha(t)\dot{x}(t)+\beta(t)\nabla f(x(t))+\epsilon(t)x(t)=0,
\end{equation*}
where $\alpha,\beta,\epsilon\colon[t_0,+\infty)\to(0,+\infty)$ are continuously differentiable functions. Strong convergence of the generated trajectories was also established under suitable assumptions. Further Tikhonov-regularized dynamical systems for unconstrained optimization problems have been investigated for viscous damping profiles of the forms $\alpha$ \cite{ref26}, $\frac{\alpha}{t}$ \cite{ref23,ref27,ref28,ref22}, $\frac{\alpha}{t^q}$ with $0<q<1$ \cite{ref13,ref29}, and $\delta\sqrt{\epsilon(t)}$ \cite{new3,ref24}.

In many practical applications, optimization problems are subject to constraints. We therefore consider the following linearly constrained optimization problem:
\begin{equation}\label{eq7}
	\begin{cases}
		\displaystyle \min_{x\in\mathcal{X}} f(x),\\
		\text{s.t. } Ax=b,
	\end{cases}
\end{equation}
where $\mathcal{X}$ and $\mathcal{Y}$ are real Hilbert spaces, $A\colon\mathcal{X}\to\mathcal{Y}$ is a continuous linear operator, and $b\in\mathcal{Y}$. To solve the problem \eqref{eq7}, one commonly introduces an augmented Lagrangian function and studies the associated primal-dual formulation. This approach makes it possible to extend second-order dynamical systems developed for unconstrained optimization problems to the linearly constrained setting. The augmented Lagrangian function $\widehat{\mathcal{L}}_\sigma\colon\mathcal{X}\times\mathcal{Y}\to\mathbb{R}$ associated with the problem \eqref{eq7} is defined by
\begin{equation*}
	\widehat{\mathcal{L}}_\sigma(x,\lambda):=f(x)+\langle\lambda,Ax-b\rangle+\frac{\sigma}{2}\|Ax-b\|^2,
\end{equation*}
where $\sigma\geq0$ is the penalty parameter.

Hulett et al. \cite{ref34} proposed the following second-order dynamical system with asymptotically vanishing damping and time-scaling:
\begin{equation*}
	\begin{cases}
		\ddot{x}(t)+\frac{\alpha}{t}\dot{x}(t)+\beta(t)\nabla_x\widehat{\mathcal{L}}_\sigma\left(x(t),\lambda(t)+\theta t\dot{\lambda}(t)\right)=0, \\
		\ddot{\lambda}(t)+\frac{\alpha}{t}\dot{\lambda}(t)-\beta(t)\nabla_\lambda\widehat{\mathcal{L}}_\sigma\left(x(t)+\theta t\dot{x}(t),\lambda(t)\right)=0.
	\end{cases}
	\tag{\text{HN-AVD}}
\end{equation*}
Here, $\alpha>0$, $\theta>0$, and $\beta\colon[t_0,+\infty)\to\mathbb{R}$ is a time-scaling function. They established fast convergence rates for the primal-dual gap, the feasibility measure, and the objective values along the generated trajectories.

Building on the work of Hulett et al. \cite{ref34}, Zhu et al. \cite{refZhuTTnew} proposed the following Tikhonov-regularized second-order primal-dual dynamical system with time-scaling and asymptotically vanishing damping:
\begin{equation*}
	\begin{cases}
		\ddot{x}(t)+\frac{\alpha}{t}\dot{x}(t)+\beta(t)\nabla_x\widehat{\mathcal{L}}_\sigma\left(x(t),\lambda(t)+\theta t\dot{\lambda}(t)\right)+\beta(t)\epsilon(t)x(t)=0, \\
		\ddot{\lambda}(t)+\frac{\alpha}{t}\dot{\lambda}(t)-\beta(t)\nabla_\lambda\widehat{\mathcal{L}}_\sigma\left(x(t)+\theta t\dot{x}(t),\lambda(t)\right)=0.
	\end{cases}
\end{equation*}
Here, the Tikhonov regularization parameter $\epsilon\colon [t_0,+\infty)\to[0,+\infty)$ is a nonincreasing continuously differentiable function satisfying $\lim\limits_{t\to+\infty}\epsilon(t)=0$. They showed that, under suitable conditions, if $\epsilon(t)$ satisfies $\int_{t_0}^{+\infty}\frac{\beta(t)\epsilon(t)}{t}\mathrm{d}t<+\infty$, then the trajectories generated by the dynamical system converge strongly to the minimum‑norm solution of problem \eqref{eq7}. On the other hand, if $\epsilon(t)$ satisfies $\int_{t_0}^{+\infty}t\beta(t)\epsilon(t)\mathrm{d}t<+\infty$, then fast convergence rates are obtained for the primal‑dual gap, the feasibility violation, the objective residual, and the norm of the objective gradient along the trajectories. Moreover, the trajectories converge weakly to a primal‑dual solution of the linearly constrained convex optimization problem \eqref{eq7}.

For further results on Tikhonov-regularized dynamical systems for linearly constrained optimization problems, we refer interested readers to \cite{shoulian,ref36,ref38,intronew1}.

Building on previous studies of unconstrained and linearly constrained optimization problems, we consider the following convex-concave bilinear saddle point problem:
\begin{equation}\label{eq13}
	\min_{x\in\mathbb{R}^n}\max_{y\in\mathbb{R}^m} \mathcal{L}(x,y):=f(x)+\langle Kx,y\rangle-g(y),
\end{equation}
where $K\in\mathbb{R}^{m\times n}$ is a linear operator from $\mathbb{R}^n$ to $\mathbb{R}^m$, $\langle\cdot,\cdot\rangle$ denotes the standard Euclidean inner product, and $f\colon\mathbb{R}^n\to\mathbb{R}$ and $g\colon\mathbb{R}^m\to\mathbb{R}$ are continuously differentiable convex functions.

Zeng et al. \cite{ref40} introduced and analyzed the following accelerated primal-dual dynamical system with vanishing damping for the problem \eqref{eq13}:
\begin{equation*}
	\begin{cases}
		\ddot{x}(t)+\frac{\alpha}{t}\dot{x}(t)+\nabla_x\mathcal{L}\left(x(t),y(t)+\theta t\dot{y}(t)\right)=0, \\
		\ddot{y}(t)+\frac{\alpha}{t}\dot{y}(t)-\nabla_y\mathcal{L}\left(x(t)+\theta t\dot{x}(t),y(t)\right)=0,
	\end{cases}
	\tag{\text{APDD}}
\end{equation*}
where $\theta=\max\left\{\frac{1}{2},\frac{3}{2\alpha}\right\}$. They further established that the primal-dual gap along the trajectories converges at the rate $\mathcal{O}\left(t^{-\frac{2\alpha}{3}}\right)$ for $0<\alpha<3$ and at the rate $\mathcal{O}\left(\frac{1}{t^2}\right)$ for $\alpha\geq3$.

He et al. \cite{ref42} proposed the following general second-order primal-dual dynamical system for convex-concave bilinear saddle point problems:
\begin{equation*}
	\begin{cases}
		\ddot{x}(t)+\alpha(t)\dot{x}(t)+\beta(t)\nabla_x\mathcal{L}\left(x(t),y(t)+\theta(t)\dot{y}(t)\right)=0, \\
		\ddot{y}(t)+\alpha(t)\dot{y}(t)-\beta(t)\nabla_y\mathcal{L}\left(x(t)+\theta(t)\dot{x}(t),y(t)\right)=0,
	\end{cases}
	\tag{\text{MPDD}}
\end{equation*}
where $\alpha,\beta,\theta \colon [t_0,+\infty)\to(0,+\infty)$ denote the damping, time-scaling, and extrapolation coefficients, respectively. They established that the primal-dual gap along the trajectories converges at the rate $\mathcal{O}\left(\frac{1}{t^{2a}\theta(t)\beta(t)}\right)$ for $a\geq0$.

Although substantial progress has been made in second-order primal-dual dynamics for convex-concave bilinear saddle point problems \cite{ref40,ref41,ref42,ref43,ref44}, a common Lyapunov-based analysis framework that simultaneously incorporates flexible time-dependent coefficients and Tikhonov-induced minimum-norm selection is still lacking. To clarify the position of this work, we first contrast our framework with the three most closely related studies along five key dimensions: damping profile, time-scaling, Tikhonov regularization, minimum-norm selection, and additional geometric assumptions.

The accelerated primal-dual dynamical system in \cite{ref40} uses the prescribed asymptotically vanishing damping profile $\frac{\alpha}{t}$ without time-scaling or Tikhonov regularization, and establishes only weak convergence of the trajectories and does not provide a mechanism for minimum-norm selection. The general framework in \cite{ref42} further allows time-dependent damping, time-scaling, and extrapolation coefficients, but does not incorporate Tikhonov regularization and hence cannot address the minimum-norm saddle point selection problem. The work in \cite{ref44} establishes strong convergence to the minimum-norm saddle point under a prescribed slowly vanishing damping profile. Its analysis employs an eventual geometric localization condition formulated directly on the generated trajectory. In contrast, our analysis introduces a moving-saddle deviation viewpoint in which the deviation between the generated trajectory and the time-dependent regularized saddle path is controlled through a path-motion condition rather than a trajectory-based localization condition.

However, understanding the role and formulation of such additional conditions in Tikhonov-regularized primal-dual dynamics remains an important issue. This raises the question of whether strong convergence can also be characterized through the evolution of the regularized saddle path and the associated coefficient conditions.

Motivated by this perspective, we develop a coefficient-compatible framework in which the viscous damping coefficient $\alpha(t)$, the time-scaling coefficient $\beta(t)$, the linear extrapolation parameter $\theta$, and the Tikhonov regularization parameter $\epsilon(t)$ are analyzed within a coupled Lyapunov-dissipation framework. The resulting framework interprets strong convergence as a tracking process toward the time-dependent regularized saddle path, thereby providing a different viewpoint for understanding the trade-off between accelerated convergence and minimum-norm solution selection.

We consider the following Tikhonov-regularized second-order primal-dual dynamical system for the problem \eqref{eq13}:
\begin{equation}\label{eq15}
	\begin{cases}
		\ddot{x}(t)+\alpha(t)\dot{x}(t)
		+\beta(t)\nabla_x\mathcal{L}_t
		\left(x(t),y(t)+\theta t\dot{y}(t)\right)=0,\\
		\ddot{y}(t)+\alpha(t)\dot{y}(t)
		-\beta(t)\nabla_y\mathcal{L}_t
		\left(x(t)+\theta t\dot{x}(t),y(t)\right)=0.
	\end{cases}
	\tag{\text{GTRD}}
\end{equation}
Here, $t\geq t_0>0$ and $\theta>0$. The functions
$\alpha:[t_0,+\infty)\to(0,+\infty)$ and
$\beta:[t_0,+\infty)\to(0,+\infty)$ denote the viscous damping and time-scaling coefficients, respectively, while $\theta t$ is the linear extrapolation coefficient.

Associated with problem \eqref{eq13}, we introduce the augmented Lagrangian function $\mathcal{L}_t:\mathbb{R}^n\times\mathbb{R}^m\to\mathbb{R}$ defined by
\begin{equation}\label{eq21}
	\begin{aligned}
		\mathcal{L}_t(x,y)
		&:=\mathcal{L}(x,y)
		+\frac{\epsilon(t)}{2}
		\left(\|x\|^2-\|y\|^2\right)\\
		&=f(x)+\langle Kx,y\rangle-g(y)
		+\frac{\epsilon(t)}{2}
		\left(\|x\|^2-\|y\|^2\right),
	\end{aligned}
\end{equation}
where the Tikhonov regularization parameter
$\epsilon:[t_0,+\infty)\to(0,+\infty)$ is assumed to be a nonincreasing $\mathcal{C}^1$ function satisfying $\lim\limits_{t\to+\infty}\epsilon(t)=0$.

The main contributions of this paper are summarized as follows:
\begin{itemize}
	\item[(1)] We develop a coefficient-compatible Lyapunov framework for second-order primal-dual dynamical systems with Tikhonov regularization. The proposed framework characterizes the differential compatibility relations among the viscous damping coefficient, the time-scaling coefficient, the linear extrapolation parameter, and the Tikhonov regularization parameter, providing a verifiable admissibility criterion for their joint parameter configurations.

    \item[(2)] By considering two cases in which the Tikhonov regularization parameter decays to zero at different rates, we establish convergence estimates for the primal-dual gap, velocities, and gradient differences. These results provide a general analytical framework for investigating the asymptotic convergence properties of the trajectories under different decay behaviors of the Tikhonov regularization parameter.

    \item[(3)] Under suitable path-motion conditions, we construct a Lyapunov function centered at the instantaneous Tikhonov-regularized saddle point and establish strong convergence of the trajectories to the minimum-norm saddle point by controlling their deviation from the regularized saddle path.

    \item[(4)] We provide explicit examples of admissible parameter configurations satisfying the proposed conditions and illustrate the theoretical results through numerical experiments.
	
\end{itemize}

The remainder of this paper is organized as follows. Section \ref{sec2} reviews basic properties of convex–concave bilinear saddle point problems and the augmented Lagrangian function, presents the auxiliary results required for the subsequent analysis, and establishes the existence and uniqueness of a global strong solution to the proposed dynamical system. Section \ref{sec3} studies the asymptotic behavior of the generated trajectories in the two cases. Convergence estimates for the primal–dual gap and velocities, together with estimates for the gradient differences and several weighted integrals, are derived, and representative parameter choices are discussed. Section \ref{sec6} establishes strong convergence of the trajectories to the minimum-norm saddle point by means of a Lyapunov function centered at the saddle point of the Tikhonov-regularized problem, and further treats the special case in which the minimum-norm saddle point is the origin. Section \ref{sec4} presents numerical experiments illustrating the theoretical results and the finite-time behavior of the proposed system alongside related primal-dual dynamical systems. Finally, Section \ref{sec5} concludes the paper.

\section{Preliminaries and Well-Posedness}\label{sec2}
This section reviews several basic properties of the Lagrangian function and the augmented Lagrangian function associated with the problem \eqref{eq13}, together with several auxiliary lemmas needed in the subsequent convergence analysis. We also introduce the notion of a global strong solution to the dynamical system \eqref{eq15} and establish its existence and uniqueness by proving the well-posedness of the corresponding Cauchy problem.

\subsection{Properties of Saddle Point Problems and Their Augmented Forms}

\begin{definition}\label{andian}
	A pair $ (x^*,y^*) \in \mathbb{R}^n \times \mathbb{R}^m $ is called a saddle point of the Lagrangian function $\mathcal{L}$ if
	\begin{equation}\label{eq19}
		\mathcal{L}(x^*, y) \leq \mathcal{L}(x^*, y^*) \leq \mathcal{L}(x, y^*),\qquad
		\forall\,(x,y)\in\mathbb{R}^n\times\mathbb{R}^m.
	\end{equation}
\end{definition}

Let $\Omega$ denote the saddle point set of $\mathcal{L}$, and assume that $\Omega$ is nonempty, closed, and convex. By the first-order optimality conditions,
\begin{equation}\label{eq20}
	(x^*,y^*)\in\Omega
	\Longleftrightarrow
	\left\{
	\begin{aligned}
		\nabla_x \mathcal{L}(x^*,y^*)&= \nabla f(x^*) + K^*y^* = 0,\\
		\nabla_y \mathcal{L}(x^*,y^*)&= -\nabla g(y^*) + Kx^* = 0.
	\end{aligned}
	\right.
\end{equation}

For the augmented Lagrangian function $\mathcal{L}_t \colon \mathbb{R}^n\times\mathbb{R}^m\to\mathbb{R}$, the partial gradients are
\begin{equation}\label{zengguang1}
	\begin{cases}
		\nabla_x \mathcal{L}_t(x,y)= \nabla f(x) + K^*y + \epsilon(t)x,\\
		\nabla_y \mathcal{L}_t(x,y)= -\nabla g(y) + Kx - \epsilon(t)y.
	\end{cases}
\end{equation}
For each $t\geq t_0>0$, the function $\mathcal{L}_t(x,\cdot)$ is $\epsilon(t)$-strongly concave for every fixed $x\in\mathbb{R}^n$, while $\mathcal{L}_t(\cdot,y)$ is $\epsilon(t)$-strongly convex for every fixed $y\in\mathbb{R}^m$. It follows that, for any $(x^*,y^*)\in\Omega$ and $t\geq t_0>0$,
\begin{equation}\label{eq31}
		\langle \nabla_x\mathcal{L}\left(x,y^*\right)+\epsilon(t)x,x^*-x \rangle \leq \mathcal{L}\left(x^*,y^*\right)-\mathcal{L}\left(x,y^*\right)+\frac{\epsilon(t)}{2}\left(\left\|x^*\right\|^2-\left\|x\right\|^2-\left\|x-x^*\right\|^2\right),
\end{equation}
and
\begin{equation*}\label{eq31'}
		\langle \nabla_y\mathcal{L}\left(x^*,y\right)-\epsilon(t)y,y^*-y \rangle \geq \mathcal{L}\left(x^*,y^*\right)-\mathcal{L}\left(x^*,y\right)+\frac{\epsilon(t)}{2}\left(\left\|y\right\|^2-\left\|y^*\right\|^2+\left\|y-y^*\right\|^2\right).
\end{equation*}
Moreover, $\mathcal{L}_t$ admits a unique Tikhonov-regularized saddle point $\left(x_t,y_t\right) \in \mathbb{R}^n \times \mathbb{R}^m$, that is,
\begin{equation}\label{eq19'}
	\mathcal{L}_t(x_t, y) \leq \mathcal{L}_t(x_t, y_t) \leq \mathcal{L}_t(x, y_t),\qquad
	\forall\,(x,y)\in\mathbb{R}^n\times\mathbb{R}^m.
\end{equation}

Throughout the paper, we work under the following assumptions.
\begin{assumption}\label{assumption}
	$f\colon\mathbb{R}^n\to\mathbb{R}$ and $g\colon\mathbb{R}^m\to\mathbb{R}$ are continuously differentiable convex functions, where $\nabla f$ and $\nabla g$ are $L_1$- and $L_2$-Lipschitz continuous, respectively, for some constants $\L_1,L_2>0$. The functions $\alpha,\beta,\epsilon \colon [t_0,+\infty)\to(0,+\infty)$ appearing in the dynamical system \eqref{eq15} belong to $\mathcal{C}^1$, and $\epsilon$ is nonincreasing with $\lim\limits_{t\to+\infty}\epsilon(t)=0$. The parameters $\theta>0$ and $t_0>0$ are constants. In addition, $K\in\mathbb{R}^{m\times n}$ is a continuous linear operator with adjoint $K^*$.
\end{assumption}

For later use in the Lyapunov analysis, we supplement Assumption \ref{assumption} with the following non-strict conditions \eqref{tiaojian2}, \eqref{tiaojian3}, \eqref{tiaojian4} and their strict counterparts \eqref{tiaojian2'}, \eqref{tiaojian3'}, \eqref{tiaojian4'}: \\
\begin{minipage}{0.48\textwidth}
	\begin{align}
		1+\theta-\theta t \alpha(t) \leq 0,
		\tag{$\mathcal{K}_1$}\label{tiaojian2}
		\\
		\left(2\theta-1\right)\beta(t)+\theta t\dot{\beta}(t) \leq 0,
		\tag{$\mathcal{K}_2$}\label{tiaojian3}
		\\
		\alpha(t)+t\dot{\alpha}(t) \leq t\beta(t)\epsilon(t).
		\tag{$\mathcal{K}_3$}\label{tiaojian4}
	\end{align}
\end{minipage}
\hfill
\begin{minipage}{0.48\textwidth}
	\begin{align}
		1+\theta-\theta t \alpha(t) < 0,
		\tag{$\mathcal{K}_1^+$}\label{tiaojian2'}
		\\
		\left(2\theta-1\right)\beta(t)+\theta t\dot{\beta}(t) < 0,
		\tag{$\mathcal{K}_2^+$}\label{tiaojian3'}
		\\
		\alpha(t)+t\dot{\alpha}(t) < t\beta(t)\epsilon(t).
		\tag{$\mathcal{K}_3^+$}\label{tiaojian4'}
	\end{align}
\end{minipage}

Conditions \eqref{tiaojian2}--\eqref{tiaojian4} constitute a set of Lyapunov-compatible differential balance conditions involving the viscous damping coefficient $\alpha(t)$, the time-scaling coefficient $\beta(t)$, the linear extrapolation parameter $\theta$, and the Tikhonov regularization parameter $\epsilon(t)$. They are introduced as sufficient conditions ensuring that the coefficients appearing in the Lyapunov derivatives have the required signs. More specifically, \eqref{tiaojian2} controls the inertial-extrapolation interaction, \eqref{tiaojian3} regulates the evolution of the time-scaling coefficient, and \eqref{tiaojian4} balances the variation of the damping coefficient with the Tikhonov regularization term. Therefore, these conditions \eqref{tiaojian2}--\eqref{tiaojian4} jointly define a verifiable sufficient admissibility regime for the parameter configurations $(\alpha,\beta,\epsilon,\theta)$ considered in this work. The strengthened conditions \eqref{tiaojian2'}--\eqref{tiaojian4'} are employed when quantitative dissipation properties are needed for the strong convergence analysis and for deriving the corresponding integral estimates.

\subsection{Auxiliary Results and Well-Posedness of the Cauchy Problem}
In this subsection, we first present several auxiliary lemmas needed for the subsequent convergence analysis and introduce the notion of a global strong solution to the dynamical system \eqref{eq15}. We then discuss the existence and uniqueness of such a solution for the corresponding Cauchy problem.

\begin{lemma}\citetext{\citealp{refnew0}; \citealp[Lemma B.3]{ref38}}\label{limit1}
	Let $t_0>0$, $\theta>0$, and $a\in\mathbb{R}^n$, and let $x \colon [t_0,+\infty)\to\mathbb{R}^n$ be continuously differentiable. Suppose that there exists a constant $D_0>0$ such that
	\begin{equation*}
		\left\|x(t)-a+\theta t\dot{x}(t)\right\|^2
		\leq D_0,
		\qquad \forall\,t\geq t_0.
	\end{equation*}
	Then $x$ is bounded on $[t_0,+\infty)$.
\end{lemma}

\begin{lemma}\label{limit}  \cite[Lemma A.3]{ref21}
   Let $a>0$, and let $\zeta \colon [a,+\infty)\to[0,+\infty)$ be a continuous function such that $\zeta\in L^1([a,+\infty))$. Let $\varphi \colon [a,+\infty)\to(0,+\infty)$ be a nondecreasing function satisfying $\lim\limits_{t\to+\infty}\varphi(t)=+\infty$. Then,
   \begin{equation*}\label{eqlimit}
   	\lim_{t\to+\infty}\frac{1}{\varphi(t)}\int_a^t\varphi(\tau)\zeta(\tau)\mathrm{d}\tau=0.
   \end{equation*}
\end{lemma}

\begin{lemma}\label{strongshoulian} \cite[Lemma 2.3]{shoulian}
	Let $\left(\hat{x}^*,\hat{y}^*\right) = \operatorname{Proj}_\Omega0$ denote the minimum-norm saddle point of $\Omega$ and let $\left(x_t,y_t\right)$ be the saddle point of $\mathcal{L}_t$. Then,
	\begin{align*}
		\lim\limits_{t\to+\infty} \|\left(x_t,y_t\right)-\left(\hat{x}^*,\hat{y}^*\right)\| = 0,
		\\
		\|\left(x_t,y_t\right)\| \leq \|\left(\hat{x}^*,\hat{y}^*\right)\|, \text{ for all } t \geq t_0>0.
	\end{align*}
\end{lemma}

\begin{lemma}\label{cunzai}
	Let $(x_t,y_t)$ be the unique saddle point of $\mathcal{L}_t$. Under Assumption \ref{assumption}, the mapping $t\mapsto(x_t,y_t)$ is locally Lipschitz continuous on $[t_0,+\infty)$. Consequently,
	\[
	(x_t,y_t)\in W_{\mathrm{loc}}^{1,\infty}\left([t_0,+\infty);\mathbb{R}^n\times\mathbb{R}^m\right),
	\]
	and $\dot{x}_t$ and $\dot{y}_t$ exist for almost every $t\geq t_0$. Moreover, for almost every $t\geq t_0$,
	\[
	\left \|\left(\dot{x}_t,\dot{y}_t\right) \right\|
	\leq
	\frac{|\dot{\epsilon}(t)|}{\epsilon(t)}\left\| \left(x_t,y_t\right) \right\|
	\leq
	\frac{|\dot{\epsilon}(t)|}{\epsilon(t)}\left\| \left(\hat{x}^*,\hat{y}^*\right) \right\|.
	\]
\end{lemma}

\begin{proof}
	By the first-order optimality conditions \eqref{zengguang1} for $(x_t,y_t)$,
	\[
	\nabla f(x_t)+K^*y_t+\epsilon(t)x_t=0,
	\qquad
	\nabla g(y_t)-Kx_t+\epsilon(t)y_t=0.
	\]
	Thus, for any $t,\tau\geq t_0$,
	\[
	\nabla f(x_t)-\nabla f(x_\tau)+K^*(y_t-y_\tau)+\epsilon(t)(x_t-x_\tau)+\left(\epsilon(t)-\epsilon(\tau)\right)x_\tau
	=0,
	\]
	and
	\[
	\nabla g(y_t)-\nabla g(y_\tau)-K(x_t-x_\tau)+\epsilon(t)(y_t-y_\tau)+\left(\epsilon(t)-\epsilon(\tau)\right)y_\tau
	=0.
	\]
	Taking the inner products with $x_t-x_\tau$ and $y_t-y_\tau$, respectively, and adding the resulting equalities, we find that the terms involving $K$ cancel. By the convexity of $f$ and $g$, the gradients $\nabla f$ and $\nabla g$ are monotone. Hence,
	\begin{equation}\label{jiexian}
	\epsilon(t)\sqrt{\|x_t-x_\tau\|^2+\|y_t-y_\tau\|^2}
	\leq
	|\epsilon(t)-\epsilon(\tau)|\sqrt{\|x_\tau\|^2+\|y_\tau\|^2}.
	\end{equation}
	
	By Lemma \ref{strongshoulian}, $(x_t,y_t)$ is bounded. Since $\epsilon\in C^1$ is positive, it is bounded away from zero and has bounded derivative on every compact interval. Hence, inequality \eqref{jiexian} implies that $t\mapsto(x_t,y_t)$ is locally Lipschitz continuous. Therefore,
	\[
	(x_t,y_t)\in W_{\mathrm{loc}}^{1,\infty}\left([t_0,+\infty);\mathbb{R}^n\times\mathbb{R}^m\right).
	\]
	
	Let $t\geq t_0$ be a point at which the mapping $t\mapsto(x_t,y_t)$ is differentiable. Dividing \eqref{jiexian} by $|t-\tau|$ and letting $\tau\to t$, we obtain
	\[
	\sqrt{\|\dot{x}_t\|^2+\|\dot{y}_t\|^2}
	\leq
	\frac{|\dot{\epsilon}(t)|}{\epsilon(t)}\sqrt{\|x_t\|^2+\|y_t\|^2}.
	\]
	Finally, by Lemma \ref{strongshoulian}, $\|(x_t,y_t)\| \leq \|(\hat{x}^*,\hat{y}^*)\|$, and therefore, for almost every $t\geq t_0$,
	\[
	\sqrt{\|\dot{x}_t\|^2+\|\dot{y}_t\|^2}
	\leq
	\frac{|\dot{\epsilon}(t)|}{\epsilon(t)}\|(\hat{x}^*,\hat{y}^*)\|.
	\]
	
	This completes the proof of Lemma \ref{cunzai}.
\end{proof}
Hence, all differential identities involving $(x_t,y_t)$ in the sequel are understood to hold almost everywhere.

We now define a global strong solution to the Cauchy problem associated with the dynamical system \eqref{eq15} and establish its existence and uniqueness.
\begin{definition}\label{definition1A1}
	Let $t_0>0$. A function $ z=(x,y) \colon [t_0,+\infty)\to\mathbb{R}^n\times\mathbb{R}^m $ is called a global strong solution of the dynamical system \eqref{eq15} if the following conditions hold:
	\begin{itemize}
		\item[(i)] $x\in W_{\mathrm{loc}}^{2,1} ([t_0,+\infty);\mathbb{R}^n)$ and $y\in W_{\mathrm{loc}}^{2,1} ([t_0,+\infty);\mathbb{R}^m)$;
		\item[(ii)] the dynamical system \eqref{eq15} is satisfied for almost every $t\in(t_0,+\infty)$;
		\item[(iii)] $x(t_0)=x_0\in\mathbb{R}^n$, $\dot{x}(t_0)=u_0\in\mathbb{R}^n$, $y(t_0)=y_0\in\mathbb{R}^m$, $\dot{y}(t_0)=v_0\in\mathbb{R}^m$.
	\end{itemize}
\end{definition}

\begin{theorem}\label{theorem1A1}
	Suppose that Assumption \ref{assumption} holds. Then, for any initial condition $\left(x(t_0),y(t_0),\dot{x}(t_0),\dot{y}(t_0)\right) \in \mathbb{R}^n\times\mathbb{R}^m \times\mathbb{R}^n\times\mathbb{R}^m$, the dynamical system \eqref{eq15} admits a unique global strong solution.
\end{theorem}

\begin{proof}
	Denote $\Lambda(t):=(\lambda_1(t),\lambda_2(t),\lambda_3(t),\lambda_4(t))=(x(t),y(t),\dot{x}(t),\dot{y}(t))$ and set $\Lambda_0:=(x_0,y_0,u_0,v_0)$. For simplicity, we write $\Lambda=(\lambda_1,\lambda_2,\lambda_3,\lambda_4)$ whenever no confusion can arise. Then, the dynamical system \eqref{eq15} can be equivalently written as the following Cauchy problem:
	\begin{equation}\label{oula}
		\begin{cases}
			\frac{\mathrm{d}\Lambda}{\mathrm{d}t}=\mathcal{F}(t,\Lambda), \\
			\Lambda(t_0)=\Lambda_0,
		\end{cases}
	\end{equation}
	where
	\begin{equation*}
		\mathcal{F}(t,\Lambda) := \begin{pmatrix} 
			\lambda_3 \\
			\lambda_4 \\
			-\alpha(t)\lambda_3-\beta(t)\left(\nabla f(\lambda_1)+\epsilon(t)\lambda_1+K^*\left(\lambda_2+\theta t\lambda_4\right)\right) \\
			-\alpha(t)\lambda_4+\beta(t)\left(-\nabla g(\lambda_2)-\epsilon(t)\lambda_2+K\left(\lambda_1+\theta t\lambda_3\right)\right)
		\end{pmatrix}.
	\end{equation*}
	
	Before deriving the required estimates, we equip the product space $\mathbb{R}^n\times\mathbb{R}^m \times\mathbb{R}^n\times\mathbb{R}^m$ with the norm $\left\|(\lambda_1,\lambda_2,\lambda_3,\lambda_4)\right\|:=\sqrt{\|\lambda_1\|^2+\|\lambda_2\|^2+\|\lambda_3\|^2+\|\lambda_4\|^2}$. Moreover, the operator norm of $K$ is defined by $\|K\|:=\sup\left\{\|Kx\|:x\in\mathbb{R}^n,\ \|x\|\leq 1\right\}$, and the operator norm of $K^*$ is defined analogously.
	
	Since $\nabla f$ and $\nabla g$ are $L_1$- and $L_2$-Lipschitz continuous, respectively, and $K$ is linear and continuous, for any $\Lambda,\Lambda^*\in\mathbb{R}^n\times\mathbb{R}^m\times\mathbb{R}^n\times\mathbb{R}^m$, we have
	\begin{equation*}\label{oula1}
		\begin{aligned}
			&\quad \left\|\mathcal{F}(t,\Lambda)-\mathcal{F}(t,\Lambda^*)\right\| \\
			&\leq \beta(t)\left(\epsilon(t)+\left\|K\right\|\right) \left\|\lambda_1-\lambda_1^*\right\| + \beta(t)\left(\epsilon(t)+\left\|K^*\right\|\right) \left\|\lambda_2-\lambda_2^*\right\| \\
			&\quad +\left(1+\alpha(t)+\theta t\beta(t)\left\|K\right\|\right)  \left\|\lambda_3-\lambda_3^*\right\| + \left(1+\alpha(t)+\theta t\beta(t)\left\|K^*\right\|\right)  \left\|\lambda_4-\lambda_4^*\right\| \\
			&\quad +\beta(t)\left\|\nabla f\left(\lambda_1\right)-\nabla f\left(\lambda_1^*\right)\right\| + \beta(t)\left\|\nabla g\left(\lambda_2\right)-\nabla g\left(\lambda_2^*\right)\right\| \\
			&\leq \beta(t)\left(\epsilon(t)+\left\|K\right\|+L_1\right) \left\|\lambda_1-\lambda_1^*\right\| + \beta(t)\left(\epsilon(t)+\left\|K^*\right\|+L_2\right) \left\|\lambda_2-\lambda_2^*\right\| \\
			&\quad +\left(1+\alpha(t)+\theta t\beta(t)\left\|K\right\|\right)  \left\|\lambda_3-\lambda_3^*\right\| + \left(1+\alpha(t)+\theta t\beta(t)\left\|K^*\right\|\right)  \left\|\lambda_4-\lambda_4^*\right\| \\
			&\leq S_1(t)\left\|\Lambda-\Lambda^*\right\|,
		\end{aligned}
	\end{equation*}
	where
	$S_1(t)=2+2\alpha(t)+\left(1+\theta t\right)\beta(t)\left(\left\|K\right\|+\left\|K^*\right\|\right)+\beta(t)\left(2\epsilon(t)+L_1+L_2\right)$.
	Since $\alpha$, $\beta$, and $\epsilon$ are continuous, $S_1 \in L^1_{loc}[t_0,+\infty)$.
	
	The Lipschitz continuity of $\nabla f$ and $\nabla g$ further gives
	\begin{equation}\label{oula7}
		\left\|\nabla f(\lambda_1)\right\|\leq\left\|\nabla f(0)\right\|+L_1\left\|\lambda_1\right\|,\qquad
		\left\|\nabla g(\lambda_2)\right\| \leq \left\|\nabla g(0)\right\|+L_2\left\|\lambda_2\right\|.
	\end{equation}
	Combining \eqref{oula} and \eqref{oula7} yields
	\begin{equation*}
		\left\|\mathcal{F}(t,\Lambda)\right\| \leq S_2(t)\left(1+\left\|\Lambda\right\|\right),
	\end{equation*}
	where $S_2(t)$ is defined by
	\begin{equation*}
		\begin{aligned}
			S_2(t)&=2+2\alpha(t)+\left(1+\theta t\right)\beta(t)\left(\left\|K\right\|+\left\|K^*\right\|\right) \\
			&\quad+\beta(t)\left(2\epsilon(t)+\|\nabla f(0)\|+\|\nabla g(0)\|+L_1+L_2\right).
		\end{aligned}
	\end{equation*}
	The continuity of $\alpha$, $\beta$, and $\epsilon$ implies that $S_2 \in L^1_{loc}[t_0,+\infty)$. Moreover, for any $\Lambda\in\mathbb{R}^n\times\mathbb{R}^m \times\mathbb{R}^n\times\mathbb{R}^m$, we have
	\begin{equation*}
		\mathcal{F}(\cdot,\Lambda)\in L^1_{loc}\left([t_0,+\infty);\mathbb{R}^n\times\mathbb{R}^m\times\mathbb{R}^n\times\mathbb{R}^m\right).
	\end{equation*}
	By \cite[Proposition 6.2.1]{ref46} and \cite[Corollary A.2]{refjia}, the Cauchy problem \eqref{oula} admits a unique global strong solution
	\begin{equation*}
		\Lambda\in W^{1,1}_{loc}\left([t_0,+\infty);\mathbb{R}^n\times\mathbb{R}^m\times\mathbb{R}^n\times\mathbb{R}^m\right).
	\end{equation*}
	Thus, the dynamical system \eqref{eq15} admits a unique global strong solution. This completes the proof.
\end{proof}

\section{Asymptotic Convergence}\label{sec3}
In this section, we investigate the asymptotic convergence properties of \eqref{eq15} under the following two decay regimes for $\epsilon(t)$: $\int_{t_0}^{+\infty} t\beta(t)\epsilon(t)\mathrm{d}t<+\infty$ and $\int_{t_0}^{+\infty}\frac{\beta(t)\epsilon(t)}{t}\mathrm{d}t<+\infty$.

\subsection{Case $ \int_{t_0}^{+\infty} t \beta(t) \epsilon(t) \mathrm{d}t < +\infty $}\label{subsec1}
In this subsection, we investigate the asymptotic convergence properties of the dynamical system \eqref{eq15} under the condition $\int_{t_0}^{+\infty}t\beta(t)\epsilon(t)\mathrm{d}t<+\infty$, which corresponds to the case in which the Tikhonov regularization parameter $\epsilon(t)$ decays to zero relatively rapidly. For any fixed $(x^*,y^*)\in\Omega$, we define the energy function $\mathcal{E} \colon [t_0,+\infty)\to[0,+\infty)$ by
\begin{equation}\label{eq23}
	\mathcal{E}(t) = \mathcal{E}_1(t) + \mathcal{E}_2(t) + \mathcal{E}_3(t),
\end{equation}
where
\begin{align*}
	\mathcal{E}_1(t) = t^2\beta(t)\left(\mathcal{L}(x(t),y^*)-\mathcal{L}(x^*,y(t))+\frac{\epsilon(t)}{2}\left(\left\|x(t)\right\|^2 +\left\| y(t)\right\|^2\right)\right),
	\\
	\mathcal{E}_2(t) = \frac{1}{2} \left\| b \left( x(t) - x^* \right) + t\dot{x}(t)\right\|^2+\frac{m(t)}{2} \left\| x(t) - x^* \right\|^2,
	\\
	\mathcal{E}_3(t) = \frac{1}{2} \left\| b \left( y(t) - y^* \right) + t \dot{y}(t)\right\|^2 +\frac{m(t)}{2} \left\| y(t) - y^* \right\|^2,
\end{align*}
where $b := \frac{1}{\theta}$, $m(t) := \frac{1}{\theta}\left(t\alpha(t)-1-\frac{1}{\theta}\right)$.
Under condition \eqref{tiaojian2}, $\mathcal{E}(t)$ is nonnegative; its boundedness will be established below under additional assumptions. Thus, $\mathcal{E}(t)$ will serve as a time-varying Lyapunov function for \eqref{eq15}.

\begin{lemma}\label{lemma1}
	Suppose that Assumption \ref{assumption} and condition \eqref{tiaojian2} hold. Then, for any trajectory $(x(t),y(t))$ of the dynamical system \eqref{eq15} and any saddle point $(x^*,y^*)\in\Omega$ of the convex-concave bilinear saddle point problem \eqref{eq13}, we have
	
	\begin{equation}\label{eq24}
		\begin{aligned}
			\dot{\mathcal{E}}(t) &\leq \left(t^2\dot{\beta}(t)+2t\beta(t)-\frac{1}{\theta}t\beta(t)\right)\left(\mathcal{L}(x(t),y^*)-\mathcal{L}(x^*,y(t))\right) \\
			&\quad + \frac{1}{2}\left(\left(t^2\dot{\beta}(t)+2t\beta(t)-\frac{1}{\theta}t\beta(t)\right)\epsilon(t) +t^2\beta(t)\dot{\epsilon}(t)\right) \left(\left\|x(t)\right\|^2+\left\|y(t)\right\|^2\right) \\
			&\quad +\frac{1}{2\theta}\left(\alpha(t)+t\dot{\alpha}(t)-t\beta(t)\epsilon(t)\right) \left(\left\|x(t)-x^*\right\|^2+\left\|y(t)-y^*\right\|^2\right) \\
			&\quad +t\left(\frac{1}{\theta}+1-t\alpha(t)\right) \left(\left\|\dot{x}(t)\right\|^2+\left\|\dot{y}(t)\right\|^2\right) + \frac{1}{2\theta}t\beta(t)\epsilon(t) \left(\left\|x^*\right\|^2+\left\|y^*\right\|^2\right).
		\end{aligned}
	\end{equation}
\end{lemma}

\begin{proof}
First, the time derivative of $\mathcal{E}_1(t)$ is given by
	\begin{equation}\label{eq25} 
		\begin{aligned}
			\quad \dot{\mathcal{E}}_1(t) =& \left(2t\beta(t)+t^2\dot{\beta}(t)\right)\left( \mathcal{L}(x(t),y^*) - \mathcal{L}(x^*,y(t)) \right) \\
			&+\frac{\epsilon(t)}{2} \left(2t\beta(t)+t^2\dot{\beta}(t)\right)\left( \left\|x(t)\right\|^2+\left\|y(t)\right\|^2 \right)\\
			&+t^2\beta(t)\bigg(\langle\nabla_x\mathcal{L}(x(t),y^*),\dot{x}(t)\rangle-\langle\nabla_y\mathcal{L}(x^*,y(t)),\dot{y}(t)\rangle  \\
			&\quad \quad \quad \quad +\frac{\dot{\epsilon}(t)}{2}\left(\left\|x(t)\right\|^2+\left\|y(t)\right\|^2\right)+\epsilon(t)\left(\langle x(t),\dot{x}(t)\rangle+\langle y(t),\dot{y}(t)\rangle\right)\bigg).
		\end{aligned}
	\end{equation}
	Next, we consider the function $ \mathcal{E}_2(t) $. Let 
	\begin{equation*}\label{eq26}
		\mu_1(t):=b \left( x(t) - x^* \right) + t\dot{x}(t) = \frac{1}{\theta}\left(x(t)-x^*\right)+t\dot{x}(t).
	\end{equation*}
	Then, 
	\begin{equation*}\label{eq27}
			\dot{\mu}_1(t) = \left(\frac{1}{\theta}+1\right)\dot{x}(t) + t\ddot{x}(t) =\left( \frac{1}{\theta}+1-t\alpha(t) \right)\dot{x}(t) - t\beta(t)\nabla_x\mathcal{L}_t\left(x(t),y(t)+\theta t\dot{y}(t)\right),
	\end{equation*}
	where the second equality follows from the first equation of the dynamical system \eqref{eq15}. Consequently,
	\begin{equation}\label{eq29}
		\begin{aligned}
			\langle\mu_1(t),\dot{\mu}_1(t)\rangle &= t\left(\frac{1}{\theta}+1-t\alpha(t)\right)\left\|\dot{x}(t)\right\|^2 + \frac{1}{\theta}\left(\frac{1}{\theta}+1-t\alpha(t)\right)\langle x(t)-x^*,\dot{x}(t) \rangle \\
			&\quad -t^2\beta(t) \langle \nabla_x\mathcal{L}(x(t),y^*)+\epsilon(t)x(t),\dot{x}(t) \rangle -t^2\beta(t) \langle K^*(y(t)-y^*+\theta t\dot{y}(t)),\dot{x}(t) \rangle \\
			&\quad -\frac{1}{\theta}t\beta(t) \langle \nabla_x\mathcal{L}(x(t),y^*)+\epsilon(t)x(t),x(t)-x^* \rangle \\
			&\quad -\frac{1}{\theta}t\beta(t) \langle K^*(y(t)-y^*+\theta t\dot{y}(t)),x(t)-x^* \rangle \\
			&\leq t\left(\frac{1}{\theta}+1-t\alpha(t)\right)\left\|\dot{x}(t)\right\|^2 + \frac{1}{\theta}\left(\frac{1}{\theta}+1-t\alpha(t)\right)\langle x(t)-x^*,\dot{x}(t) \rangle \\
			&\quad -t^2\beta(t) \langle \nabla_x\mathcal{L}(x(t),y^*)+\epsilon(t)x(t),\dot{x}(t) \rangle -t^2\beta(t) \langle K^*(y(t)-y^*+\theta t\dot{y}(t)),\dot{x}(t) \rangle \\
			&\quad -\frac{1}{\theta}t\beta(t) \left(\mathcal{L}(x(t),y^*)-\mathcal{L}(x^*,y^*) +\frac{\epsilon(t)}{2}\left(\|x(t)\|^2-\|x^*\|^2+\|x(t)-x^*\|^2\right)\right) \\
			&\quad -\frac{1}{\theta}t\beta(t) \langle K^*(y(t)-y^*+\theta t\dot{y}(t)),x(t)-x^* \rangle,
		\end{aligned}
	\end{equation}
	where we have used \eqref{eq31}, together with the following identity obtained directly from \eqref{zengguang1}:
	\begin{equation*}\label{eq30}
		\nabla_x\mathcal{L}_t\left(x(t),y(t)+\theta t\dot{y}(t)\right) = \nabla_x\mathcal{L}\left(x(t),y^*\right)+\epsilon(t)x(t)+K^*\left(y(t)-y^*+\theta t\dot{y}(t)\right).
	\end{equation*}
	
	Moreover,
	\begin{equation}\label{eq35}
		\begin{aligned}
			\frac{\mathrm{d}}{\mathrm{d}t}\left( \frac{m(t)}{2}\left\|x(t)-x^*\right\|^2 \right) &= \frac{\dot{m}(t)}{2}\left\|x(t)-x^*\right\|^2+m(t)\langle x(t)-x^*,\dot{x}(t)\rangle \\
			&= \frac{\alpha(t)+t\dot{\alpha}(t)}{2\theta}\|x(t)-x^*\|^2+\frac{1}{\theta}\left(t\alpha(t)-1-\frac{1}{\theta}\right)\langle x(t)-x^*,\dot{x}(t) \rangle.
		\end{aligned}
	\end{equation}
	Combining \eqref{eq29} and \eqref{eq35}, we obtain
	\begin{equation}\label{eq36}
		\begin{aligned}
			\dot{\mathcal{E}}_2(t) &= \langle \mu_1(t),\dot{\mu}_1(t)\rangle +\frac{\mathrm{d}}{\mathrm{d}t}\left( \frac{m(t)}{2}\left\|x(t)-x^*\right\|^2 \right) \\
			&\leq -\frac{1}{\theta}t\beta(t) \left( \mathcal{L}(x(t),y^*)-\mathcal{L}(x^*,y^*)+\frac{\epsilon(t)}{2}\left(\left\|x(t)\right\|^2-\left\|x^*\right\|^2\right) \right) \\
			&\quad +\frac{1}{2\theta}\left(\alpha(t)+t\dot{\alpha}(t)-t\beta(t)\epsilon(t)\right)\left\|x(t)-x^*\right\|^2 \\
			&\quad +t\left(\frac{1}{\theta}+1-t\alpha(t)\right)\left\|\dot{x}(t)\right\|^2 - t^2\beta(t) \langle \nabla_x\mathcal{L}(x(t),y^*)+\epsilon(t)x(t),\dot{x}(t) \rangle \\
			&\quad - \langle K^*(y(t)-y^*+\theta t\dot{y}(t)),\frac{1}{\theta}t\beta(t)(x(t)-x^*)+t^2\beta(t)\dot{x}(t) \rangle .
		\end{aligned}
	\end{equation}
	
	Similarly, we have
	\begin{equation}\label{eq37}
		\begin{aligned}
			\dot{\mathcal{E}}_3(t) &\leq -\frac{1}{\theta}t\beta(t) \left( \mathcal{L}(x^*,y^*)-\mathcal{L}(x^*,y(t))+\frac{\epsilon(t)}{2}\left(\left\|y(t)\right\|^2-\left\|y^*\right\|^2\right) \right) \\
			&\quad + \frac{1}{2\theta}\left(\alpha(t)+t\dot{\alpha}(t)-t\beta(t)\epsilon(t)\right)\left\|y(t)-y^*\right\|^2 \\
			&\quad + t\left(\frac{1}{\theta}+1-t\alpha(t)\right)\left\|\dot{y}(t)\right\|^2 + t^2\beta(t) \langle \nabla_y\mathcal{L}(x^*,y(t))-\epsilon(t)y(t),\dot{y}(t) \rangle \\
			&\quad + \langle K(x(t)-x^*+\theta t\dot{x}(t)),\frac{1}{\theta}t\beta(t)(y(t)-y^*)+t^2\beta(t)\dot{y}(t) \rangle .
		\end{aligned}
	\end{equation}
	
	Combining \eqref{eq23}, \eqref{eq25}, \eqref{eq36}, and \eqref{eq37} yields
	\begin{equation*}\label{eq38}
		\begin{aligned}
			\dot{\mathcal{E}}(t) &\leq \left(t^2\dot{\beta}(t)+2t\beta(t)-\frac{1}{\theta}t\beta(t)\right)\left(\mathcal{L}(x(t),y^*)-\mathcal{L}(x^*,y(t))\right) \\
			&\quad +\left(\frac{\epsilon(t)}{2}\left(t^2\dot{\beta}(t)+2t\beta(t)-\frac{1}{\theta}t\beta(t)\right) +\frac{\dot{\epsilon}(t)}{2}t^2\beta(t)\right) \left(\|x(t)\|^2+\|y(t)\|^2\right) \\
			&\quad +\frac{1}{2\theta}\left(\alpha(t)+t\dot{\alpha}(t)-t\beta(t)\epsilon(t)\right) \left(\|x(t)-x^*\|^2+\|y(t)-y^*\|^2\right) \\
			&\quad +t\left(\frac{1}{\theta}+1-t\alpha(t)\right) \left(\|\dot{x}(t)\|^2+\|\dot{y}(t)\|^2\right) + \frac{1}{2\theta}t\beta(t)\epsilon(t) \left(\|x^*\|^2+\|y^*\|^2\right).
		\end{aligned}
	\end{equation*}

	This completes the proof of Lemma \ref{lemma1}.
\end{proof}

\begin{theorem}\label{theorem1}
	Suppose that Assumption \ref{assumption} and conditions \eqref{tiaojian2}--\eqref{tiaojian4} hold. Assume further that $\lim\limits_{t\to+\infty}t^2\beta(t)=+\infty$ and $\int_{t_0}^{+\infty}t\beta(t)\epsilon(t)\mathrm{d}t<+\infty$. Then, for any trajectory $(x(t),y(t))$ of the dynamical system \eqref{eq15} and any saddle point $(x^*,y^*)\in\Omega$ of the convex-concave bilinear saddle point problem \eqref{eq13}, the trajectory $(x(t),y(t))$ is bounded, and the following statements hold:
	
	\begin{equation*}
		\mathcal{L}\left(x(t),y^*\right) - 	\mathcal{L}\left(x^*,y(t)\right) = \mathcal{O}\left( \dfrac{1}{t^2\beta(t)} \right),
	\end{equation*}
	\begin{equation*}
		\left\| \dot{x}(t) \right\| = \mathcal{O}\left( \frac{1}{t} \right), \qquad \left\| \dot{y}(t) \right\| = \mathcal{O}\left( \frac{1}{t} \right),
	\end{equation*}
	\begin{equation*}
		\left\| \nabla f(x(t)) - \nabla f(x^*) \right\| = \mathcal{O}\left( \frac{1}{t \sqrt{\beta(t)}} \right), \qquad \left\| \nabla g(y(t)) - \nabla g(y^*) \right\| = \mathcal{O}\left( \frac{1}{t \sqrt{\beta(t)}} \right).
	\end{equation*}
    In particular, if \eqref{tiaojian2'} holds for all $t\geq t_0$, then
    \begin{equation*}
    	\int_{t_0}^{+\infty} t\left(t\alpha(t)-\frac{1}{\theta}-1\right)\left(\left\|\dot{x}(t)\right\|^2+\left\|\dot{y}(t)\right\|^2\right) \mathrm{d}t < +\infty.
    \end{equation*}
    Furthermore, if \eqref{tiaojian3'} holds, then
    \begin{equation*}
    	\int_{t_0}^{+\infty} \left( (1 - 2\theta)\beta(t) - \theta t \dot{\beta}(t) \right) t \left(\mathcal{L}\left(x(t),y^*\right)-\mathcal{L}\left(x^*,y(t)\right)\right) \mathrm{d}t < +\infty,
    \end{equation*}
    \begin{equation*}
    	\int_{t_0}^{+\infty} \left( (1 - 2\theta)\beta(t) - \theta t \dot{\beta}(t) \right) t \left\| \nabla f(x(t)) - \nabla f(x^*) \right\|^2 \mathrm{d}t < +\infty,
    \end{equation*}
    and
    \begin{equation*}
    	\int_{t_0}^{+\infty} \left( (1 - 2\theta)\beta(t) - \theta t \dot{\beta}(t) \right) t \left\| \nabla g(y(t)) - \nabla g(y^*) \right\|^2 \mathrm{d}t < +\infty.
    \end{equation*}
\end{theorem}

\begin{proof}
 	By Lemma \ref{lemma1}, integrating \eqref{eq24} from $t_0$ to $t$ yields
 	\begin{equation}\label{eq50}
 		\begin{aligned}
 			&\mathcal{E}(t) -\int_{t_0}^{t} \left(\tau^2\dot{\beta}(\tau)+2\tau\beta(\tau)-\frac{1}{\theta}\tau\beta(\tau)\right) \left(\mathcal{L}(x(\tau),y^*)-\mathcal{L}(x^*,y(\tau))\right) \mathrm{d}\tau \\
 			&\quad - \frac{1}{2} \int_{t_0}^{t} \left(\left(\tau^2\dot{\beta}(\tau)+2\tau\beta(\tau)-\frac{1}{\theta}\tau\beta(\tau)\right)\epsilon(\tau) + \tau^2\beta(\tau)\dot{\epsilon}(\tau)\right) \left(\left\|x(\tau)\right\|^2+\left\|y(\tau)\right\|^2\right) \mathrm{d}\tau \\
 			&\quad - \frac{1}{2\theta} \int_{t_0}^{t} \left(\tau\dot{\alpha}(\tau)+\alpha(\tau)-\tau\beta(\tau)\epsilon(\tau)\right) \left(\left\|x(\tau)-x^*\right\|^2+\left\|y(\tau)-y^*\right\|^2\right) \mathrm{d}\tau \\
 			&\quad - \int_{t_0}^{t} \tau\left(\frac{1}{\theta}+1-\tau\alpha(\tau)\right) \left(\left\|\dot{x}(\tau)\right\|^2+\left\|\dot{y}(\tau)\right\|^2\right) \mathrm{d}\tau \\
 			&\leq \mathcal{E}(t_0) + \int_{t_0}^{t} \frac{1}{2\theta}\tau\beta(\tau)\epsilon(\tau) \left(\left\|x^*\right\|^2+\left\|y^*\right\|^2\right) \mathrm{d}\tau.
 		\end{aligned}
 	\end{equation}
 	Under conditions \eqref{tiaojian2}, \eqref{tiaojian3}, and \eqref{tiaojian4}, the following inequalities hold:
 	\begin{flalign}
 		&\ (a)\, 2t\beta(t)+t^2\dot{\beta}(t)-\frac{1}{\theta}t\beta(t) \leq 0;&\nonumber\\
 		&\ (b)\, t\dot{\alpha}(t)+\alpha(t)-t\beta(t)\epsilon(t) \leq 0;&\nonumber\\
 		&\ (c)\, \left(2t\beta(t)+t^2\dot{\beta}(t)-\frac{1}{\theta}t\beta(t)\right)\epsilon(t)+t^2\beta(t)\dot{\epsilon}(t) \leq 0;&\nonumber\\
 		&\ (d)\, t\left(\frac{1}{\theta}+1-t\alpha(t)\right) \leq 0.&\nonumber
 	\end{flalign}
    Inequalities $(a)$ and $(b)$ follow from \eqref{tiaojian3} and \eqref{tiaojian4}, respectively. Since $\epsilon(t)\geq 0$ and $\dot{\epsilon}(t)\leq 0$, inequality $(c)$ holds, while inequality $(d)$ follows from \eqref{tiaojian2}. In addition, \eqref{eq19} gives $\mathcal{L}(x(t),y^*)-\mathcal{L}(x^*,y(t))\geq0$. Together with the integrability condition $\int_{t_0}^{+\infty}t\beta(t)\epsilon(t)\mathrm{d}t<+\infty$ and inequality \eqref{eq50}, the above sign properties imply that $\mathcal{E}(t)$ is bounded on $[t_0,+\infty)$. Thus, $\mathcal{E}(t)$ serves as a Lyapunov function in the present analysis. In particular, there exists a constant $D_1\geq 0$ such that
 	
 	\begin{equation*}
 		\mathcal{E}(t) \leq \mathcal{E}(t_0) + \int_{t_0}^{t} \frac{1}{2\theta}\tau\beta(\tau)\epsilon(\tau) \left(\|x^*\|^2+\|y^*\|^2\right) \mathrm{d}\tau \leq D_1, \quad \forall t \geq t_0.
 	\end{equation*}
 	By \eqref{eq23}, $\left\|\frac{1}{\theta}\left(x(t)-x^*\right)+t\dot{x}(t)\right\|$ and $\left\|\frac{1}{\theta}\left(y(t)-y^*\right)+t\dot{y}(t)\right\|$ are bounded for all $t\geq t_0$. Applying Lemma \ref{limit1} then shows that the trajectory $(x(t),y(t))$ is bounded on $[t_0,+\infty)$.
 	
 	Note that
 	\begin{equation*}
 		\left\|t\dot{x}(t)\right\|^2 \leq 2\left\|\frac{1}{\theta}\left(x(t)-x^*\right)\right\|^2 + 2\left\|\frac{1}{\theta}\left(x(t)-x^*\right)+t\dot{x}(t)\right\|^2.
 	\end{equation*}
 	The boundedness of $\left\|\frac{1}{\theta}\left(x(t)-x^*\right)+t\dot{x}(t)\right\|$, together with the boundedness of the trajectory $(x(t),y(t))$, yields
 	\begin{equation*}
 		\left\| \dot{x}(t) \right\| = \mathcal{O}\left( \frac{1}{t} \right).
 	\end{equation*}
 	The same argument yields
 	\begin{equation*}
 		\left\| \dot{y}(t) \right\| = \mathcal{O}\left( \frac{1}{t} \right).
 	\end{equation*}
 	Moreover, the definition of $\mathcal{E}(t)$ gives
 	\begin{equation*}
 		t^2\beta(t) \left(\mathcal{L}\left(x(t),y^*\right) - \mathcal{L}\left(x^*,y(t)\right)\right) \leq \mathcal{E}(t).
 	\end{equation*}
 	The boundedness of $\mathcal{E}(t)$ yields
 	\begin{equation}\label{eq59}
 		\mathcal{L}\left(x(t),y^*\right) - \mathcal{L}\left(x^*,y(t)\right) = \mathcal{O}\left(\frac{1}{t^2\beta(t)}\right).
 	\end{equation}
 	Since $\lim\limits_{t\to+\infty}t^2\beta(t)=+\infty$, it follows that the primal-dual gap converges to zero. Moreover, let $L_{max}=\max\{L_1,L_2\}$. Since $f$ and $g$ are convex and $\nabla f$ and $\nabla g$ are $L_1$- and $L_2$-Lipschitz continuous, respectively, we have
 	
 	\begin{equation*}
 		f(x(t)) \geq f(x^*) + \langle \nabla f(x^*), x(t) - x^* \rangle + \frac{1}{2L_{max}} \left\|\nabla f(x(t)) - \nabla f(x^*)\right\|^2
 	\end{equation*}
 	and
 	\begin{equation*}
 		g(y(t)) \geq g(y^*) + \langle \nabla g(y^*), y(t) - y^* \rangle + \frac{1}{2L_{max}} \left\|\nabla g(y(t)) - \nabla g(y^*)\right\|^2.
 	\end{equation*}
 	Using the definition of $\mathcal{L}$ in \eqref{eq13}, we have
 	\begin{equation}\label{new1}
 		\begin{aligned}
 			\mathcal{L}\left(x(t), y^*\right) - \mathcal{L}\left(x^*, y(t)\right) &= f(x(t)) - f(x^*) + g(y(t)) - g(y^*) + \langle Kx(t), y^* \rangle - \langle Kx^*, y(t) \rangle  \\
 			&\geq \langle \nabla f(x^*), x(t) - x^* \rangle + \frac{1}{2L_{max}} \left\| \nabla f(x(t)) - \nabla f(x^*) \right\|^2 + \langle Kx(t),y^* \rangle \\
 			&\quad - \langle \nabla g(y^*), y^* - y(t) \rangle + \frac{1}{2L_{max}} \left\| \nabla g(y(t)) - \nabla g(y^*) \right\|^2 - \langle Kx^*,y(t) \rangle \\
 			&= \frac{1}{2L_{max}} \left\| \nabla f(x(t)) - \nabla f(x^*) \right\|^2 + \frac{1}{2L_{max}} \left\| \nabla g(y(t)) - \nabla g(y^*) \right\|^2,
 		\end{aligned}
 	\end{equation}
    where the last equality follows from \eqref{eq20}. Combining \eqref{eq59} with \eqref{new1} gives
    \begin{equation*}
 	\left\| \nabla f(x(t)) - \nabla f(x^*) \right\| = \mathcal{O}\left( \frac{1}{t \sqrt{\beta(t)}} \right),\qquad \left\| \nabla g(y(t)) - \nabla g(y^*) \right\| = \mathcal{O}\left( \frac{1}{t \sqrt{\beta(t)}} \right).
    \end{equation*}
    
    Moreover, \eqref{eq50} gives
 	\begin{equation}\label{eq60}
 		\begin{aligned}
 			&\mathcal{E}(t) -\int_{t_0}^{t} \left(\tau^2\dot{\beta}(\tau)+2\tau\beta(\tau)-\frac{1}{\theta}\tau\beta(\tau)\right) \left(\mathcal{L}(x(\tau),y^*)-\mathcal{L}(x^*,y(\tau))\right) \mathrm{d}\tau \\
 			&\quad - \int_{t_0}^{t} \tau\left(\frac{1}{\theta}+1-\tau\alpha(\tau)\right) \left(\left\|\dot{x}(\tau)\right\|^2+\left\|\dot{y}(\tau)\right\|^2\right) \mathrm{d}\tau \\
 			&\leq \mathcal{E}(t_0)+ \int_{t_0}^{t} \frac{1}{2\theta}\tau\beta(\tau)\epsilon(\tau) \left(\left\|x^*\right\|^2+\left\|y^*\right\|^2\right) \mathrm{d}\tau.
 		\end{aligned}
 	\end{equation}
 	If condition \eqref{tiaojian2} is strengthened to \eqref{tiaojian2'}, then \eqref{eq60}, together with $\int_{t_0}^{+\infty}t\beta(t)\epsilon(t)\mathrm{d}t<+\infty$ and the nonnegativity of $\mathcal{E}(t)$ for $t\geq t_0$, gives
    \begin{equation*}
    	\int_{t_0}^{+\infty} t\left(t\alpha(t)-\frac{1}{\theta}-1\right) \left( \left\| \dot{x}(t) \right\|^2 + \left\| \dot{y}(t) \right\|^2 \right) \mathrm{d}t < +\infty.
    \end{equation*}
    Furthermore, if \eqref{tiaojian3} is strengthened to \eqref{tiaojian3'}, then \eqref{eq60} yields
    \begin{equation}\label{new2}
    	\int_{t_0}^{+\infty} \left( (1 - 2\theta)\beta(t) - \theta t \dot{\beta}(t) \right) t \left(\mathcal{L}\left(x(t),y^*\right)-\mathcal{L}\left(x^*,y(t)\right)\right) \mathrm{d}t < +\infty.
    \end{equation}
    Combining \eqref{new1} with \eqref{new2} yields
    \begin{equation*}
    	\int_{t_0}^{+\infty} \left( (1 - 2\theta)\beta(t) - \theta t \dot{\beta}(t) \right) t \left\| \nabla f(x(t)) - \nabla f(x^*) \right\|^2 \mathrm{d}t < +\infty
    \end{equation*}
    and
    \begin{equation*}
    	\int_{t_0}^{+\infty} \left( (1 - 2\theta)\beta(t) - \theta t \dot{\beta}(t) \right) t \left\| \nabla g(y(t)) - \nabla g(y^*) \right\|^2 \mathrm{d}t < +\infty.
    \end{equation*}

 	This completes the proof of Theorem \ref{theorem1}.
\end{proof}

\begin{remark}\label{remarkalpha}
	If the Tikhonov regularization term is formally removed by setting $\epsilon(t)\equiv0$, the proposed dynamical system \eqref{eq15} reduces, under suitable choices of the remaining coefficients, to several classical unregularized models. In particular, after suppressing the dual coupling, one recovers the accelerated vanishing-damping dynamical system (AVD$_\alpha$) proposed by Su et al. \cite{ref9} for unconstrained optimization problems. Moreover, by taking $\beta(t)\equiv1$, $\theta=\frac{1}{2}$, and $\alpha(t)=\frac{\alpha}{t}$ with $\alpha\geq3$, one recovers the convergence result of Theorem 1 in \cite{ref40}. Likewise, by setting $\alpha(t)=\frac{\alpha}{t}$, one recovers the corresponding conclusion of Theorem~6 in \cite{ref42}.
\end{remark}

\subsection{Case $ \int_{t_0}^{+\infty} \frac{\beta(t)\epsilon(t)}{t} \mathrm{d}t < +\infty $}\label{subsec2}
In this subsection, we investigate the asymptotic convergence properties of the dynamical system \eqref{eq15} under the condition $\int_{t_0}^{+\infty}\frac{\beta(t)\epsilon(t)}{t}\mathrm{d}t<+\infty$, which corresponds to the case in which the Tikhonov regularization parameter $\epsilon(t)$ decays to zero more slowly. For any fixed $(x^*,y^*) \in \Omega$, we define the energy function $ \bar{\mathcal{E}}(t) $ by
\begin{equation}\label{eq62}
	\bar{\mathcal{E}}(t) = \bar{\mathcal{E}}_1(t) + \bar{\mathcal{E}}_2(t) + \bar{\mathcal{E}}_3(t),
\end{equation}
where
\begin{align*}
	\bar{\mathcal{E}}_1(t) = \beta(t) \left(\mathcal{L}\left(x(t),y^*\right)-\mathcal{L}\left(x^*,y(t)\right)+\frac{\epsilon(t)}{2}\left(\left\|x(t)\right\|^2+\left\|y(t)\right\|^2\right)\right), 
	\\
	\bar{\mathcal{E}}_2(t) = \frac{1}{2} \left\| \bar{b}(t)\left(x(t)-x^*\right)+\dot{x}(t) \right\|^2 + \frac{\bar{m}(t)}{2}\left\|x(t)-x^*\right\|^2,
	\\
	\bar{\mathcal{E}}_3(t) = \frac{1}{2} \left\| \bar{b}(t)\left(y(t)-y^*\right)+\dot{y}(t) \right\|^2 + \frac{\bar{m}(t)}{2}\left\|y(t)-y^*\right\|^2,
\end{align*}
with $\bar{b}(t) := \frac{1}{\theta t}$, $\bar{m}(t) = \frac{\theta t\alpha(t)-\theta-1}{\theta^2t^2}$.

\begin{lemma}\label{lemma2}
	Suppose that Assumption \ref{assumption} and condition \eqref{tiaojian2} hold. Then, for any global trajectory $(x(t),y(t))$ of the dynamical system \eqref{eq15} and any saddle point $(x^*,y^*)\in\Omega$ of the convex-concave bilinear saddle point problem \eqref{eq13}, we have
    \begin{equation}\label{eq75}
	    \begin{aligned}
		    \frac{2}{t}\bar{\mathcal{E}}(t) + \dot{\bar{\mathcal{E}}}(t) &\leq \left(\frac{2}{t}\beta(t)+\dot{\beta}(t)-\frac{1}{\theta t}\beta(t)\right) \left(\mathcal{L}(x(t),y^*)-\mathcal{L}(x^*,y(t))\right) \\
		    &\quad +\frac{1}{2}\left(\left(\frac{2}{t}\beta(t)+\dot{\beta}(t)-\frac{1}{\theta t}\beta(t)\right)\epsilon(t)+\beta(t)\dot{\epsilon}(t)\right) \left(\left\|x(t)\right\|^2+\left\|y(t)\right\|^2\right) \\
		    &\quad +\frac{1}{2\theta t}\left(\frac{t\dot{\alpha}(t)+\alpha(t)}{t}-\beta(t)\epsilon(t)\right) \left(\left\|x(t)-x^*\right\|^2+\left\|y(t)-y^*\right\|^2\right) \\
		    &\quad +\left(\frac{1+\theta}{\theta t}-\alpha(t)\right) \left(\left\|\dot{x}(t)\right\|^2+\left\|\dot{y}(t)\right\|^2\right) + \frac{1}{2\theta t}\beta(t)\epsilon(t)\left(\left\|x^*\right\|^2+\left\|y^*\right\|^2\right) .
	    \end{aligned}
    \end{equation}
\end{lemma}

\begin{proof}
    Observe that
	\begin{equation*}
		\begin{aligned}
			\frac{1}{2}\left\| \bar{b}(t)\left(x(t)-x^*\right) + \dot{x}(t) \right\|^2 &= \frac{1}{2}\bar{b}(t)^2\left\|x(t)-x^*\right\|^2 + \frac{1}{2}\left\|\dot{x}(t)\right\|^2 + \bar{b}(t)\langle x(t)-x^*,\dot{x}(t) \rangle \\
			&= \frac{1}{2\theta^2t^2}\left\|x(t)-x^*\right\|^2 + \frac{1}{2}\left\|\dot{x}(t)\right\|^2 + \frac{1}{\theta t}\langle x(t)-x^*,\dot{x}(t) \rangle.
		\end{aligned}
	\end{equation*}
	Hence,
	\begin{equation}\label{eq66}
		\bar{\mathcal{E}}_2(t) = \frac{t\alpha(t)-1}{2\theta t^2}\left\|x(t)-x^*\right\|^2 + \frac{1}{2}\left\|\dot{x}(t)\right\|^2 + \frac{1}{\theta t}\langle x(t)-x^*,\dot{x}(t)\rangle .
	\end{equation}
	Similarly,
	\begin{equation}\label{eq67}
		\bar{\mathcal{E}}_3(t) = \frac{t\alpha(t)-1}{2\theta t^2}\left\|y(t)-y^*\right\|^2 + \frac{1}{2}\left\|\dot{y}(t)\right\|^2 + \frac{1}{\theta t}\langle y(t)-y^*,\dot{y}(t)\rangle .
	\end{equation}
	Equations \eqref{eq62}, \eqref{eq66}, and \eqref{eq67} yield the equivalent representation
	\begin{equation}\label{eq70}
		\begin{aligned}
			\bar{\mathcal{E}}(t) &= \beta(t)\left(\mathcal{L}(x(t),y^*)-\mathcal{L}(x^*,y(t))\right) +\frac{1}{2}\beta(t)\epsilon(t)\left(\left\|x(t)\right\|^2+\left\|y(t)\right\|^2\right) \\
			&\quad +\frac{t\alpha(t)-1}{2\theta t^2} \left(\left\|x(t)-x^*\right\|^2+\left\|y(t)-y^*\right\|^2\right) \\
			&\quad + \frac{1}{2}\left(\left\|\dot{x}(t)\right\|^2+\left\|\dot{y}(t)\right\|^2\right) + \frac{1}{\theta t}(\langle x(t)-x^*,\dot{x}(t)\rangle+\langle y(t)-y^*,\dot{y}(t)\rangle).			
		\end{aligned}
	\end{equation}
	Differentiating $\bar{\mathcal{E}}(t)$ with respect to $t$ gives
	\begin{equation}\label{eq71}
		\begin{aligned}
			\dot{\bar{\mathcal{E}}}(t) &= \dot{\beta}(t) \left(\mathcal{L}(x(t),y^*)-\mathcal{L}(x^*,y(t))\right)+\left(\frac{1}{\theta t}-\alpha(t)\right)\left(\left\|\dot{x}(t)\right\|^2+\left\|\dot{y}(t)\right\|^2\right) \\
			&\quad +\frac{1}{2}\left(\dot{\beta}(t)\epsilon(t)+\beta(t)\dot{\epsilon}(t)\right)\left(\left\|x(t)\right\|^2+\left\|y(t)\right\|^2\right) \\
			&\quad +\frac{t^2\dot{\alpha}(t)-t\alpha(t)+2}{2\theta t^3} (\left\|x(t)-x^*\right\|^2+\left\|y(t)-y^*\right\|^2) \\
			&\quad -\frac{2}{\theta t^2} \left(\langle x(t)-x^*,\dot{x}(t)\rangle + \langle y(t)-y^*,\dot{y}(t)\rangle \right) \\
			&\quad -\beta(t) \left( \langle K^*(y(t)-y^*+\theta t\dot{y}(t)),\dot{x}(t) \rangle - \langle K(x(t)-x^*+\theta t\dot{x}(t)),\dot{y}(t) \rangle \right) \\
			&\quad -\frac{1}{\theta t}\beta(t) \left( \langle \nabla_x\mathcal{L}_t\left(x(t),y^*\right),x(t)-x^*\rangle + \langle \nabla_y\mathcal{L}_t\left(x^*,y(t)\right),y(t)-y^*\rangle \right) \\
			&\quad -\frac{1}{\theta t}\beta(t) \left( \langle K^*(y(t)-y^*+\theta t\dot{y}(t)),x(t)-x^* \rangle - \langle K(x(t)-x^*+\theta t\dot{x}(t)),y(t)-y^* \rangle \right) .
		\end{aligned}
	\end{equation}
	Using the same argument as in the proof of Lemma \ref{lemma1} together with \eqref{eq70} and \eqref{eq71}, we obtain
	\begin{equation*}\label{eq74}
		\begin{aligned}
			\frac{2}{t}\bar{\mathcal{E}}(t) + \dot{\bar{\mathcal{E}}}(t) &\leq \frac{1}{2\theta t}\beta(t)\epsilon(t)\left(\left\|x^*\right\|^2+\left\|y^*\right\|^2\right) +\left(\frac{1}{\theta t}-\alpha(t)+\frac{1}{t}\right) \left(\left\|\dot{x}(t)\right\|^2+\left\|\dot{y}(t)\right\|^2\right) \\
			&\quad +\left(\frac{2}{t}\beta(t)+\dot{\beta}(t)-\frac{1}{\theta t}\beta(t)\right) \left(\mathcal{L}(x(t),y^*)-\mathcal{L}(x^*,y(t))\right) \\
			&\quad +\frac{1}{2}\left(\left(\frac{2}{t}\beta(t)+\dot{\beta}(t)-\frac{1}{\theta t}\beta(t)\right)\epsilon(t)+\beta(t)\dot{\epsilon}(t)\right) \left(\left\|x(t)\right\|^2+\left\|y(t)\right\|^2\right) \\
			&\quad +\frac{1}{2\theta t}\left(\frac{t\dot{\alpha}(t)+\alpha(t)}{t}-\beta(t)\epsilon(t)\right) \left(\left\|x(t)-x^*\right\|^2+\left\|y(t)-y^*\right\|^2\right).
		\end{aligned}
	\end{equation*}
	
	This completes the proof of Lemma \ref{lemma2}.
\end{proof}

\begin{theorem}\label{theorem2}
	Suppose that Assumption \ref{assumption} and conditions \eqref{tiaojian2}--\eqref{tiaojian4} hold. If
	$\int_{t_0}^{+\infty}\frac{\beta(t)\epsilon(t)}{t}\mathrm{d}t<+\infty$, then, for any trajectory $(x(t),y(t))$ of the dynamical system \eqref{eq15} and any saddle point $(x^*,y^*)\in\Omega$ of the convex-concave bilinear saddle point problem \eqref{eq13}, the following statements hold:
	\begin{enumerate}
		\item[(1)] If there exists a positive constant $D_2$ such that $D_2\leq\theta t\alpha(t)-\theta-1$ for all $t\geq t_0$, then
		\[
		\lim\limits_{t\to+\infty}\left\|\dot{x}(t)\right\|=0,
		\qquad
		\lim\limits_{t\to+\infty}\left\|\dot{y}(t)\right\|=0.
		\]
		
		\item[(2)] If $ \liminf\limits_{t\to+\infty}\beta(t)>0$, then
		\[
		\lim_{t\to+\infty}\left(\mathcal{L}\left(x(t),y^*\right)-\mathcal{L}\left(x^*,y(t)\right)\right)=0.
		\]
		In particular, if $\lim\limits_{t\to+\infty}\beta(t)=+\infty$, then
		\[
		\mathcal{L}\left(x(t),y^*\right)-\mathcal{L}\left(x^*,y(t)\right)=o\left(\frac{1}{\beta(t)}\right),
		\]
		and
		\[
		\left\|\nabla f(x(t))-\nabla f(x^*)\right\|=o\left(\frac{1}{\sqrt{\beta(t)}}\right),
		\qquad
		\left\|\nabla g(y(t))-\nabla g(y^*)\right\|=o\left(\frac{1}{\sqrt{\beta(t)}}\right).
		\]
	\end{enumerate}
\end{theorem}

\begin{proof}
	Under the conditions on $\alpha(t)$, $\beta(t)$, and $\epsilon(t)$ in Theorem \ref{theorem2}, we have
	\begin{flalign}
		&\ (a)\, \frac{2}{t}\beta(t)+\dot{\beta}(t)-\frac{1}{\theta t}\beta(t) \leq 0;&\nonumber\\
		&\ (b)\, \frac{1}{2\theta t}\left(\frac{t\dot{\alpha}(t)+\alpha(t)}{t}-\beta(t)\epsilon(t)\right) \leq 0;&\nonumber\\
		&\ (c)\, \frac{1}{2}\left(\left(\frac{2}{t}\beta(t)+\dot{\beta}(t)-\frac{1}{\theta t}\beta(t)\right)\epsilon(t)+\beta(t)\dot{\epsilon}(t)\right) \leq 0;&\nonumber\\
		&\ (d)\, \frac{1+\theta}{\theta t}-\alpha(t) \leq 0.&\nonumber
	\end{flalign}
	Inequalities $(a)$ and $(b)$ follow from \eqref{tiaojian3} and \eqref{tiaojian4}, respectively. Inequality $(c)$ follows from $(a)$, the positivity of $\beta(t)$ and $\epsilon(t)$, and the fact that $\epsilon(t)$ is nonincreasing, while $(d)$ follows from \eqref{tiaojian2}. Substituting these inequalities into \eqref{eq75} gives, for all $t\geq t_0$,
	\begin{equation}\label{eq78}
		\frac{2}{t}\bar{\mathcal{E}}(t)+\dot{\bar{\mathcal{E}}}(t)
		\leq
		\frac{1}{2\theta t}\beta(t)\epsilon(t)\left(\left\|x^*\right\|^2+\left\|y^*\right\|^2\right).
	\end{equation}
	Multiplying \eqref{eq78} by $t^2$ and integrating over $[t_0,t]$ gives
	\begin{equation}\label{eq79}
		\bar{\mathcal{E}}(t)
		\leq
		\frac{t_0^2\bar{\mathcal{E}}(t_0)}{t^2}+\frac{\left\|x^*\right\|^2+\left\|y^*\right\|^2}{2\theta t^2}\int_{t_0}^{t}\tau\beta(\tau)\epsilon(\tau)\mathrm{d}\tau.
	\end{equation}
	Applying Lemma \ref{limit} with $a=t_0$, $\varphi(t)=t^2$, and $\zeta(t)=\frac{\beta(t)\epsilon(t)}{t}$ yields
	\begin{equation}\label{eq80}
		\lim_{t\to+\infty}\frac{1}{t^2}\int_{t_0}^{t}\tau\beta(\tau)\epsilon(\tau)\mathrm{d}\tau=0,
	\end{equation}
	where the condition $\int_{t_0}^{+\infty}\frac{\beta(t)\epsilon(t)}{t}\,\mathrm{d}t<+\infty$ has been used. Since $\bar{\mathcal{E}}(t)\geq0$, it follows from \eqref{eq79} and \eqref{eq80} that
	\begin{equation*}
		\lim_{t\to+\infty}\bar{\mathcal{E}}(t)=0.
	\end{equation*}
	It follows from the definition of the Lyapunov function $\bar{\mathcal{E}}(t)$ that
	\begin{align}
		\lim_{t\to+\infty}\left\|\frac{1}{\theta t}\left(x(t)-x^*\right)+\dot{x}(t)\right\|=0,
		\label{eq74'}
		\\
		\lim_{t\to+\infty}\frac{\theta t\alpha(t)-\theta-1}{2\theta^2t^2}\left\|x(t)-x^*\right\|^2=0.
		\label{eq76'}
	\end{align}
	Under the additional assumption in statement (1), there exists a constant $D_2$ such that
	\begin{equation*}
		D_2 \leq \theta t\alpha(t)-\theta-1.
	\end{equation*}
	Hence,
	\begin{equation*}
		0 \leq \frac{D_2}{2\theta^2t^2}\left\|x(t)-x^*\right\|^2 \leq \frac{\theta t\alpha(t)-\theta-1}{2\theta^2t^2}\left\|x(t)-x^*\right\|^2.
	\end{equation*}
	Letting $t \to +\infty$ in the above inequality and using \eqref{eq76'}, the squeeze theorem yields
	\begin{equation}\label{eq77'}
		\lim_{t\to+\infty}\left\|\frac{1}{\theta t}\left(x(t)-x^*\right)\right\|=0.
	\end{equation}
	Note that
	\begin{equation}\label{new3'}
		\left\|\dot{x}(t)\right\|^2 \leq 2\left\|\frac{1}{\theta t}\left(x(t)-x^*\right)\right\|^2+2\left\|\frac{1}{\theta t}\left(x(t)-x^*\right)+\dot{x}(t)\right\|^2.
	\end{equation}
	Combining \eqref{eq74'}, \eqref{eq77'}, and \eqref{new3'} yields
	\begin{equation*}
		\lim_{t\to+\infty}\left\|\dot{x}(t)\right\|=0.
	\end{equation*}
	The same argument yields
	\begin{equation*}
		\lim_{t\to+\infty}\left\|\dot{y}(t)\right\|=0.
	\end{equation*}
	
	Moreover, the definition of $\bar{\mathcal{E}}(t)$ gives
	\begin{equation*}
		\lim_{t\to+\infty}
		\beta(t)\left(\mathcal{L}\left(x(t),y^*\right)-\mathcal{L}\left(x^*,y(t)\right)\right)=0.
	\end{equation*}
	If $\liminf\limits_{t\to+\infty}\beta(t)>0$, this implies that
	\begin{equation*}
		\lim_{t\to+\infty}
		\left(\mathcal{L}\left(x(t),y^*\right)-\mathcal{L}\left(x^*,y(t)\right)\right)=0.
	\end{equation*}
	In particular, if $\lim\limits_{t\to+\infty}\beta(t)=+\infty$, then
	\begin{equation}\label{eq83}
		\mathcal{L}\left(x(t),y^*\right)-\mathcal{L}\left(x^*,y(t)\right)=o\left(\frac{1}{\beta(t)}\right).
	\end{equation}
	Finally, \eqref{new1} and \eqref{eq83}, together with the positivity of $\beta(t)$, yield
	\begin{equation*}
		\left\|\nabla f(x(t))-\nabla f(x^*)\right\|
		= o\left(\frac{1}{\sqrt{\beta(t)}}\right),
		\qquad
		\left\|\nabla g(y(t))-\nabla g(y^*)\right\|
		= o\left(\frac{1}{\sqrt{\beta(t)}}\right).
	\end{equation*}

	This completes the proof of Theorem \ref{theorem2}.
\end{proof}

Under the second weighted integrability condition, corresponding to slower Tikhonov decay, the present framework also covers the subcritical damping regime $\alpha(t)=\alpha/t^q$, where $0<q<1$.
\begin{corollary}\label{cor:q_damping}
	Let $t_0=1$ and
	\[
	\alpha(t)=\frac{\alpha}{t^q}, \qquad \beta(t)=t^\beta, \qquad \epsilon(t)=\frac{c}{t^r},
	\]
	where $0<q<1$, $\beta\geq 0$, $\alpha>\beta+3$, $c\geq \alpha(1-q)$, $\beta<r\leq \beta+q+1$. Choose $\frac{1}{\alpha-1}<\theta\leq\frac{1}{\beta+2}$. Then, for every fixed saddle point
	$(x^*,y^*)\in\Omega$,
	\[
	\lim_{t\to+\infty}\|\dot{x}(t)\|=0, \qquad \lim_{t\to+\infty}\|\dot{y}(t)\|=0,\qquad \lim_{t\to+\infty}\left(\mathcal{L}(x(t),y^*)-\mathcal{L}(x^*,y(t))\right)=0.
	\]
	If $\beta>0$, then
	\[
	L(x(t),y^*)-L(x^*,y(t))=o\left(t^{-\beta}\right),
	\]
	and
	\[
	\|\nabla f(x(t))-\nabla f(x^*)\|=o\left(t^{-\frac{\beta}{2}}\right), \qquad \|\nabla g(y(t))-\nabla g(y^*)\|=o\left(t^{-\frac{\beta}{2}}\right).
	\]
\end{corollary}

\subsection{Particular Cases}\label{subsec5}
To illustrate the scope and verifiability of the differential compatibility conditions, we next consider several representative families of coefficient functions. These examples show that the resulting admissible parameter class includes not only the classical inverse-time damping profile but also perturbations of power-law profiles and genuinely non-power damping coefficients. Specifically, we let $\beta(t)=t^\beta$ ($\beta \geq 0$), and consider the following two cases according to the sign of $\alpha(t)+t\dot{\alpha}(t)$. \\
\textbf{Case 1:} $\bm{\alpha(t)+t\dot{\alpha}(t) \leq 0}$

In this case, let $\alpha(t)=\frac{\alpha}{t}+\frac{c}{t^p}$ and $\beta(t)=t^\beta$. Theorems \ref{theorem1} and \ref{theorem2} then lead to the following result.

\begin{theorem}\label{teli1}
	For the dynamical system \eqref{eq15}, let
	\begin{equation*}
		\alpha(t)=\frac{\alpha}{t}+\frac{c}{t^p},\qquad \beta(t)=t^\beta,\qquad \frac{1}{\alpha-1} < \theta < \frac{1}{\beta+2},
	\end{equation*}
	where $c$, $p$, $\alpha$, and $\beta$ are constants satisfying $c \geq 0$, $p>1$, $\alpha \geq 3$, $\alpha > \beta+3$, and $\beta \geq 0$. We also take $t_0 \geq 1$. Suppose that $\epsilon \colon [t_0,+\infty) \to (0,+\infty)$ is a nonincreasing $\mathcal{C}^1$ function satisfying $\lim\limits_{t\to+\infty}\epsilon(t)=0$, and let $\left(x(t),y(t)\right)$ be a global solution of the dynamical system \eqref{eq15}. Then, for any fixed $\left(x^*,y^*\right) \in \Omega$, the following statements hold: \\
	$(1)$\quad If $\int_{t_0}^{+\infty} t^{\beta + 1} \epsilon(t) \mathrm{d}t < +\infty$, then the trajectory $\left(x(t),y(t)\right)$ is bounded, and the following estimates hold:
	\[
	\begin{aligned}
		&\textbf{\emph{(Pointwise Estimates)}}
		\quad
		\left\{
		\begin{aligned}
			&\mathcal{L}\left(x(t),y^*\right)-\mathcal{L}\left(x^*,y(t)\right)=\mathcal{O}\left(t^{-\beta-2}\right), \\
			&\left\|\dot{x}(t)\right\|=\mathcal{O}\left(\frac{1}{t}\right),\quad \left\|\dot{y}(t)\right\|=\mathcal{O}\left(\frac{1}{t}\right), \\
			&\left\| \nabla f(x(t)) - \nabla f(x^*) \right\| = \mathcal{O}\left( t^{-\frac{\beta+2}{2}} \right), \\
			&\left\| \nabla g(y(t)) - \nabla g(y^*) \right\|= \mathcal{O}\left( t^{-\frac{\beta+2}{2}} \right); 
		\end{aligned}
		\right.
		\\[1ex]
		&\textbf{\emph{(Integral Estimates)}}
		\quad
		\left\{
		\begin{aligned}
			&\int_{t_0}^{+\infty} t \left( \| \dot{x}(t) \|^2 + \| \dot{y}(t) \|^2 \right) \mathrm{d}t < +\infty, \\
			&\int_{t_0}^{+\infty} t^{\beta+1} \left(\mathcal{L}\left(x(t),y^*\right)-\mathcal{L}\left(x^*,y(t)\right)\right) \mathrm{d}t < +\infty, \\
			&\int_{t_0}^{+\infty} t^{\beta+1} \left\| \nabla f(x(t)) - \nabla f(x^*) \right\|^2 \mathrm{d}t < +\infty, \\
			&\int_{t_0}^{+\infty} t^{\beta+1} \left\| \nabla g(y(t)) - \nabla g(y^*) \right\|^2 \mathrm{d}t < +\infty.
		\end{aligned}
		\right.
	\end{aligned}
	\]
	$(2)$\quad If $\int_{t_0}^{+\infty} t^{\beta - 1} \epsilon(t) \mathrm{d}t < +\infty$, then
	\begin{itemize}
		\item[(a)] $\lim\limits_{t \to \infty} \left\|\dot{x}(t)\right\| = 0, \; \lim\limits_{t \to \infty} \left\|\dot{y}(t)\right\| = 0; $
		\item[(b)] $\lim\limits_{t \to +\infty} \mathcal{L}\left(x(t),y^*\right) - \mathcal{L}\left(x^*,y(t)\right) = 0$. In particular, when $\beta>0$, we have $\mathcal{L}\left(x(t),y^*\right) - \mathcal{L}\left(x^*,y(t)\right) = o\left(t^{-\beta}\right)$, $\left\| \nabla f(x(t)) - \nabla f(x^*) \right\| = o\left(t^{-\frac{\beta}{2}}\right)$ and $\left\| \nabla g(y(t)) - \nabla g(y^*) \right\| = o\left(t^{-\frac{\beta}{2}}\right).$
	\end{itemize}
\end{theorem}

\begin{proof}
	A direct calculation gives
	\begin{align*}
		\alpha(t)+t\dot{\alpha}(t) = -\frac{(p-1)c}{t^p} \leq 0,
		\\
		(2\theta - 1) t \beta(t) + \theta t^2 \dot{\beta}(t) = \left((2 + \beta)\theta - 1\right) t^{\beta + 1} \leq 0.
	\end{align*}
	Together with $\frac{1}{\alpha-1} < \theta < \frac{1}{\beta+2}$, these identities show that, for $t \geq t_0 \geq 1$, the parameters satisfy conditions \eqref{tiaojian2'}, \eqref{tiaojian3'}, and \eqref{tiaojian4}.
	
	When $\int_{t_0}^{+\infty} t^{\beta + 1} \epsilon(t) \mathrm{d}t < +\infty$, Theorem \ref{theorem1} yields $\int_{t_0}^{+\infty} t\left(t\alpha(t)-\frac{1}{\theta}-1\right)\left(\left\|\dot{x}(t)\right\|^2+\left\|\dot{y}(t)\right\|^2\right) \mathrm{d}t < +\infty$. Together with $\frac{\alpha}{t} \leq \alpha(t)$, this estimate yields $\int_{t_0}^{+\infty} t \left( \| \dot{x}(t) \|^2 + \| \dot{y}(t) \|^2 \right) \mathrm{d}t < +\infty$. If $\int_{t_0}^{+\infty} t^{\beta - 1} \epsilon(t) \mathrm{d}t < +\infty$ and $\beta \geq 0$, it follows that $\liminf_{t \to +\infty} t^\beta > 0$. When $\beta > 0$, $\lim\limits_{t \to +\infty} \beta(t) = +\infty$. Moreover, $\theta t\alpha(t)-\theta-1 = \theta\alpha+\theta c t^{1-p}-\theta-1 \geq \theta(\alpha-1)-1>0$. Hence, the condition in Theorem \ref{theorem2} holds with $D_2=\theta(\alpha-1)-1$. Consequently, the conclusions of Theorem \ref{teli1} follow from Theorems \ref{theorem1} and \ref{theorem2}.
\end{proof}

\begin{remark}
	The family $\alpha(t)=\frac{\alpha}{t}+\frac{c}{t^p}$ includes the classical vanishing damping coefficient as a special case and also covers a broader class of viscous damping coefficients. When $c=0$, it reduces to the classical coefficient $\alpha(t)=\frac{\alpha}{t}$, for which $\alpha(t)+t\dot{\alpha}(t)=0$. For $c>0$ and $p>1$, however, $\alpha(t)+t\dot{\alpha}(t)=-\frac{(p-1)c}{t^p}<0$.
\end{remark}

%

\begin{flushleft}
	\textbf{Case 2:} $\bm{0 < \alpha(t)+t\dot{\alpha}(t) \leq t\beta(t)\epsilon(t)}$
	
	To demonstrate that the class of damping coefficients allowed by condition \eqref{tiaojian4} is not restricted to pure power laws or finite sums of power-law terms, we next consider a logarithmically corrected coefficient satisfying $0<\alpha(t)+t\dot{\alpha}(t)\leq t\beta(t)\epsilon(t)$.
\end{flushleft}

\begin{theorem}\label{teli2}
	For the dynamical system \eqref{eq15}, let
	\[
	\alpha(t)=\frac{\alpha-(\log t)^{-\sigma}}{t},
	\qquad
	\beta(t)=t^\beta,
	\qquad
	\epsilon(t)=\frac{1}{t^{\beta+2}(\log t)^r},
	\qquad
	\frac{1}{\alpha-1}<\theta<\frac{1}{\beta+2},
	\]
	where $\alpha>\beta+3$, $\beta\geq0$, $\sigma>0$, and $1<r<\sigma+1$. Let $t_0>1$ be sufficiently large, and let $(x(t),y(t))$ be a global solution of the dynamical system \eqref{eq15} on $[t_0,+\infty)$. Then, for any $(x^*,y^*)\in\Omega$, the trajectory $(x(t),y(t))$ is bounded, and the following estimates hold:
	\[
	\begin{aligned}
		&\textbf{\emph{(Pointwise Estimates)}}
		\quad
		\left\{
		\begin{aligned}
			&\mathcal{L}(x(t),y^*)-\mathcal{L}(x^*,y(t))
			=\mathcal{O}\left(t^{-\beta-2}\right),\\
			&\|\dot{x}(t)\|=\mathcal{O}\left(\frac{1}{t}\right),
			\qquad
			\|\dot{y}(t)\|=\mathcal{O}\left(\frac{1}{t}\right),\\
			&\|\nabla f(x(t))-\nabla f(x^*)\|
			=\mathcal{O}\left(t^{-\frac{\beta+2}{2}}\right),\\
			&\|\nabla g(y(t))-\nabla g(y^*)\|
			=\mathcal{O}\left(t^{-\frac{\beta+2}{2}}\right);
		\end{aligned}
		\right.
		\\[1ex]
		&\textbf{\emph{(Integral Estimates)}}
		\quad
		\left\{
		\begin{aligned}
			&\int_{t_0}^{+\infty}
			\left(
			(\alpha\theta-\theta-1)t
			-\frac{\theta t}{(\log t)^\sigma}
			\right)
			\left(
			\|\dot{x}(t)\|^2+\|\dot{y}(t)\|^2
			\right)\,\mathrm{d}t<+\infty,\\
			&\int_{t_0}^{+\infty}
			t^{\beta+1}
			\left(
			\mathcal{L}(x(t),y^*)-\mathcal{L}(x^*,y(t))
			\right)\,\mathrm{d}t<+\infty,\\
			&\int_{t_0}^{+\infty}
			t^{\beta+1}
			\|\nabla f(x(t))-\nabla f(x^*)\|^2\,\mathrm{d}t<+\infty,\\
			&\int_{t_0}^{+\infty}
			t^{\beta+1}
			\|\nabla g(y(t))-\nabla g(y^*)\|^2\,\mathrm{d}t<+\infty.
		\end{aligned}
		\right.
	\end{aligned}
	\]
	Moreover, $\lim\limits_{t\to+\infty}\|\dot{x}(t)\|=0$, $\lim\limits_{t\to+\infty}\|\dot{y}(t)\|=0$, $\lim\limits_{t\to+\infty}
	\left(\mathcal{L}(x(t),y^*)-\mathcal{L}(x^*,y(t))\right)=0$. In particular, when $\beta>0$,
	\[
	\mathcal{L}(x(t),y^*)-\mathcal{L}(x^*,y(t))=o\left(t^{-\beta}\right),
	\]
	and
	\[
	\|\nabla f(x(t))-\nabla f(x^*)\|=o\left(t^{-\frac{\beta}{2}}\right),
	\qquad
	\|\nabla g(y(t))-\nabla g(y^*)\|=o\left(t^{-\frac{\beta}{2}}\right).
	\]
\end{theorem}

\begin{proof}
	A direct calculation gives
	\[
	\alpha(t)+t\dot{\alpha}(t)=\frac{\sigma}{t(\log t)^{\sigma+1}}>0,
	\qquad
	t\beta(t)\epsilon(t)=\frac{1}{t(\log t)^r},
	\]
	and hence
	\[
	\frac{\alpha(t)+t\dot{\alpha}(t)}{t\beta(t)\epsilon(t)}=\sigma(\log t)^{r-\sigma-1}\longrightarrow 0.
	\]
	Moreover,
	\[
	1+\theta-\theta t\alpha(t)\longrightarrow1+\theta-\alpha\theta<0,
	\qquad
	(2\theta-1)\beta(t)+\theta t\dot{\beta}(t)=\left((\beta+2)\theta-1\right)t^\beta<0.
	\]
	Therefore, since $r<\sigma+1$ and $\frac{1}{\alpha-1}<\theta<\frac{1}{\beta+2}$, conditions \eqref{tiaojian2'}, \eqref{tiaojian3'}, and \eqref{tiaojian4'} hold for all $t\ge t_0$, provided that $t_0$ is sufficiently large.
	
	Furthermore, since $r>1$,
	\[
	\int_{t_0}^{+\infty}t\beta(t)\epsilon(t)\,\mathrm{d}t
	=
	\int_{t_0}^{+\infty}\frac{\mathrm{d}t}{t(\log t)^r}
	<+\infty,
	\qquad
	\int_{t_0}^{+\infty}\frac{\beta(t)\epsilon(t)}{t}\,\mathrm{d}t
	=
	\int_{t_0}^{+\infty}\frac{\mathrm{d}t}{t^3(\log t)^r}
	<+\infty.
	\]
	Also,
	\[
	t^2\beta(t)=t^{\beta+2}\longrightarrow+\infty,
	\qquad
	\liminf_{t\to+\infty}\beta(t)>0,
	\]
	and $\beta(t)=t^\beta\to+\infty$ whenever $\beta>0$.
	Finally,
	\[
	\theta t\alpha(t)-\theta-1
	\longrightarrow
	\theta(\alpha-1)-1>0,
	\]
	so that, for $t_0$ sufficiently large, the additional condition in
	Theorem \ref{theorem2} is satisfied for some $D_2>0$.
	Consequently, the conclusions follow directly from Theorems
	\ref{theorem1} and \ref{theorem2}.
\end{proof}

\section{Strong Convergence of Trajectories}\label{sec6}
In this section, we study the strong convergence of the trajectories generated by \eqref{eq15} toward the minimum-norm saddle point of problem \eqref{eq13}. Since the Tikhonov-regularized saddle point $(x_t,y_t)$ varies with time, we formulate the analysis in terms of the deviation between the generated trajectory and the time-dependent regularized saddle path, and construct a Lyapunov functional centered at $(x_t,y_t)$. By controlling this deviation and using the convergence property $(x_t,y_t)\rightarrow(\hat{x}^{*},\hat{y}^{*})$, we establish strong convergence to the minimum-norm saddle point.

\begin{theorem}\label{theorem3}
	Let $(x_t,y_t)$ denote the unique saddle point of $\mathcal{L}_t$, let $(x(t),y(t))$ be a global solution of the dynamical system \eqref{eq15}, and let $\left(\hat{x}^*,\hat{y}^*\right)=\operatorname{Proj}_{\Omega}0$ be the unique minimum-norm saddle point of $\Omega$. Suppose that $t_0 \geq 1$, that Assumption \ref{assumption} and conditions \eqref{tiaojian2'}, \eqref{tiaojian3}, and \eqref{tiaojian4'} hold, and that $\lim\limits_{t\to+\infty}t^2\beta(t)\epsilon(t)=+\infty$. Assume further that
	\begin{equation}\tag{$\mathcal{N}_1$}\label{N2}
		\lim\limits_{t \to +\infty} \frac{\int_{t_0}^{t}\tau^2\mathcal{R}(\tau)\mathrm{d}\tau}{t^2\beta(t)\epsilon(t)} = 0,
	\end{equation}
	where 
	\begin{align*}
		\mathcal{R}(t) = \frac{\|P_x(t)\|^2 + \|P_y(t)\|^2}{l(t)} + \frac{\|\dot{x}_t\|^2 + \|\dot{y}_t\|^2}{2\theta^2t^2\Gamma(t)},
		\\
		P_x(t) = -\beta(t)K^*\dot{y}_t - \beta(t)\dot{\epsilon}(t)x_t + \frac{\alpha(t)-t^{-1}}{\theta t}\dot{x}_t,
		\\
		P_y(t) = \beta(t)K\dot{x}_t - \beta(t)\dot{\epsilon}(t)y_t + \frac{\alpha(t)-t^{-1}}{\theta t}\dot{y}_t,
		\\
		l(t)=\frac{1}{\theta t}\left(\beta(t)\epsilon(t)-\dot{\alpha}(t)-\frac{\alpha(t)}{t}\right),
		\\
		\Gamma(t)=\alpha(t)-\frac{1+\theta}{\theta t}.
	\end{align*}
	By \eqref{tiaojian2'} and \eqref{tiaojian4'}, we have $\Gamma(t) > 0$ and $l(t) > 0$ for all $t \geq t_0$. Then, the trajectory converges strongly to the minimum-norm saddle point $(\hat{x}^*,\hat{y}^*)$, that is,
	\begin{equation*}
	\lim_{t\to+\infty} \left\|(x(t),y(t))-(\hat{x}^*,\hat{y}^*)\right\| = 0.
	\end{equation*}
	Note that the quantities $P_x(t)$, $P_y(t)$, and $\mathcal{R}(t)$ are defined almost everywhere. Moreover, by Lemma \ref{cunzai} and the continuity and positivity of $l(t)$ and $\Gamma(t)$, the function $\mathcal{R}$ is measurable and belongs to $L^1_{\mathrm{loc}}([t_0,+\infty))$.
\end{theorem}

\begin{proof}
    Define $\widehat{\mathcal{E}} \colon [t_0,+\infty)\to[0,+\infty)$ by
	\begin{equation}\label{eqnew}
		\begin{aligned}
			\widehat{\mathcal{E}}(t):=
			&\beta(t)\left( \mathcal{L}_t\left(x(t),y_t\right)-\mathcal{L}_t\left(x_t,y(t)\right) \right) \\
			&+\frac{1}{2}\left\| \frac{1}{\theta t} \left(x(t)-x_t\right)+\dot{x}(t) \right\|^2 +\frac{1}{2}\left\| \frac{1}{\theta t} \left(y(t)-y_t\right)+\dot{y}(t) \right\|^2 \\
			&+\frac{\theta t\alpha(t)-\theta-1}{2\theta^2t^2} \left(\|x(t)-x_t\|^2 + \|y(t)-y_t\|^2\right).
		\end{aligned}
	\end{equation}
	By Theorem \ref{theorem1A1}, Lemma \ref{cunzai}, and Assumption \ref{assumption}, we have $(x,y)\in W^{2,1}_{\mathrm{loc}}$ and $(x_t,y_t)\in W^{1,\infty}_{\mathrm{loc}}$. Hence, $\widehat{\mathcal{E}}\in W^{1,1}_{\mathrm{loc}}([t_0,+\infty))$ and is therefore locally absolutely continuous. All differential identities and inequalities below are understood to hold for almost every $t\ge t_0$.
	
	Observe that
	\begin{equation*}
		\frac{1}{2}\left\| \frac{1}{\theta t}\left(x(t)-x_t\right)+\dot{x}(t) \right\|^2 = \frac{1}{2\theta^2t^2}\|x(t)-x_t\|^2 + \frac{1}{\theta t}\langle x(t)-x_t,\dot{x}(t) \rangle + \frac{1}{2}\|\dot{x}(t)\|^2
	\end{equation*}
	and
	\begin{equation*}
		\frac{1}{2}\left\| \frac{1}{\theta t}\left(y(t)-y_t\right)+\dot{y}(t) \right\|^2 = \frac{1}{2\theta^2t^2}\|y(t)-y_t\|^2 + \frac{1}{\theta t}\langle y(t)-y_t,\dot{y}(t) \rangle + \frac{1}{2}\|\dot{y}(t)\|^2.
	\end{equation*}
	Thus,
		\begin{equation}\label{eqnew'}
		\begin{aligned}
			\widehat{\mathcal{E}}(t) =
			&\beta(t)\left( \mathcal{L}_t\left(x(t),y_t\right)-\mathcal{L}_t\left(x_t,y(t)\right) \right) \\
			&+\frac{t\alpha(t)-1}{2\theta t^2} \left(\|x(t)-x_t\|^2 + \|y(t)-y_t\|^2\right) + \frac{1}{2} \left( \|\dot{x}(t)\|^2 + \|\dot{y}(t)\|^2 \right) \\
			&+\frac{1}{\theta t} \left( \langle x(t)-x_t,\dot{x}(t) \rangle  + \langle y(t)-y_t,\dot{y}(t) \rangle \right).
		\end{aligned}
	\end{equation}
	By the $\epsilon(t)$-strong convexity of $\mathcal{L}_t\left(\cdot,y_t\right)$ and the $\epsilon(t)$-strong concavity of $\mathcal{L}_t\left(x_t,\cdot\right)$, we have
	\begin{align*}
		\langle \nabla_x\mathcal{L}_t\left(x(t),y_t\right),x_t-x(t) \rangle \leq \mathcal{L}_t\left(x_t,y_t\right)-\mathcal{L}_t\left(x(t),y_t\right)-\frac{\epsilon(t)}{2}\|x(t)-x_t\|^2,
		\\
		\langle \nabla_y\mathcal{L}_t\left(x_t,y(t)\right),y_t-y(t) \rangle \geq \mathcal{L}_t\left(x_t,y_t\right)-\mathcal{L}_t\left(x_t,y(t)\right)+\frac{\epsilon(t)}{2}\|y(t)-y_t\|^2.
	\end{align*}
	Differentiating $\widehat{\mathcal{E}}(t)$ and using the dynamical system \eqref{eq15}, the saddle-point optimality conditions for $(x_t,y_t)$, the above inequalities, and $\dot{\epsilon}(t) \leq 0$, we obtain
	\begin{equation}\label{strong'}
		\begin{aligned}
			\dot{\widehat{\mathcal{E}}}(t)\leq & \left(\dot{\beta}(t)-\frac{\beta(t)}{\theta t}\right) \left(\mathcal{L}_t\left(x(t),y_t\right) - \mathcal{L}_t\left(x_t,y(t)\right)\right) \\
			&-\frac{1}{2}\left(\frac{\beta(t)\epsilon(t)}{\theta t}-\frac{\dot{\alpha}(t)}{\theta t}+\frac{\alpha(t)}{\theta t^2}-\frac{2}{\theta t^3}\right)\left(\|x(t)-x_t\|^2+\|y(t)-y_t\|^2\right) \\
			&+ \left(\frac{1}{\theta t}-\alpha(t)\right)\left(\|\dot{x}(t)\|^2+\|\dot{y}(t)\|^2\right) - \frac{2}{\theta t^2}\left( \langle x(t)-x_t,\dot{x}(t) \rangle  + \langle y(t)-y_t,\dot{y}(t) \rangle \right) \\
			&-\langle x(t)-x_t,P_x(t) \rangle - \langle y(t)-y_t,P_y(t) \rangle - \frac{1}{\theta t}\left(\langle \dot{x}(t),\dot{x}_t \rangle + \langle \dot{y}(t),\dot{y}_t \rangle\right).
		\end{aligned}
	\end{equation}
	Equations \eqref{eqnew'} and \eqref{strong'} together imply that
	\begin{equation}\label{strong}
	\begin{aligned}
		\frac{2}{t}\widehat{\mathcal{E}}(t)+\dot{\widehat{\mathcal{E}}}(t)\leq & \left(\dot{\beta}(t)+\left(\frac{2}{t}-\frac{1}{\theta t}\right)\beta(t)\right) \left(\mathcal{L}_t\left(x(t),y_t\right) - \mathcal{L}_t\left(x_t,y(t)\right)\right) \\
		&-\frac{l(t)}{2}\left(\|x(t)-x_t\|^2+\|y(t)-y_t\|^2\right) - \Gamma(t)\left(\|\dot{x}(t)\|^2+\|\dot{y}(t)\|^2\right) \\
		&-\langle x(t)-x_t,P_x(t) \rangle - \langle y(t)-y_t,P_y(t) \rangle - \frac{1}{\theta t}\left(\langle \dot{x}(t),\dot{x}_t \rangle + \langle \dot{y}(t),\dot{y}_t \rangle\right).
	\end{aligned}
	\end{equation}
	Young's inequality, together with \eqref{tiaojian2'} and \eqref{tiaojian4'}, gives
	\begin{align}
		-\langle x(t)-x_t,P_x(t) \rangle \leq \frac{l(t)}{4}\|x(t)-x_t\|^2 + \frac{\|P_x(t)\|^2}{l(t)},
		\label{Young1}
		\\[5pt]
	    -\langle y(t)-y_t,P_y(t) \rangle \leq \frac{l(t)}{4}\|y(t)-y_t\|^2 + \frac{\|P_y(t)\|^2}{l(t)},
		\label{Young2}
		\\[5pt]
		-\frac{1}{\theta t}\left(\langle \dot{x}(t),\dot{x}_t \rangle + \langle \dot{y}(t),\dot{y}_t \rangle\right) \leq \frac{\Gamma(t)}{2}\left(\|\dot{x}(t)\|^2+\|\dot{y}(t)\|^2\right) + \frac{\|\dot{x}_t\|^2 + \|\dot{y}_t\|^2}{2\theta^2t^2\Gamma(t)}.
		\label{Young3}
	\end{align}
	Substituting \eqref{Young1}, \eqref{Young2}, and \eqref{Young3} into \eqref{strong} gives
	\begin{equation*}
		\begin{aligned}
			\frac{2}{t}\widehat{\mathcal{E}}(t)+\dot{\widehat{\mathcal{E}}}(t)\leq & \left(\dot{\beta}(t)+\left(\frac{2}{t}-\frac{1}{\theta t}\right)\beta(t)\right) \left(\mathcal{L}_t\left(x(t),y_t\right) - \mathcal{L}_t\left(x_t,y(t)\right)\right) \\
			&-\frac{l(t)}{4}\left(\|x(t)-x_t\|^2+\|y(t)-y_t\|^2\right) \\
			&-\frac{\Gamma(t)}{2}\left(\|\dot{x}(t)\|^2+\|\dot{y}(t)\|^2\right) \\
			&+\mathcal{R}(t).
		\end{aligned}
	\end{equation*}
	In view of \eqref{eq19'}, \eqref{tiaojian2'}, \eqref{tiaojian3}, and \eqref{tiaojian4'}, this reduces to
	\begin{equation*}
			\frac{2}{t}\widehat{\mathcal{E}}(t)+\dot{\widehat{\mathcal{E}}}(t) \leq \mathcal{R}(t).
	\end{equation*}
	Multiplying the above inequality by $t^2$ and integrating over $[t_0,t]$ gives
	\begin{equation}\label{lyapunov}
		\widehat{\mathcal{E}}(t) \leq \frac{t_0^2\widehat{\mathcal{E}}(t_0)}{t^2} + \frac{1}{t^2} \int_{t_0}^{t}\tau^2\mathcal{R}(\tau)\mathrm{d}\tau.
	\end{equation}
	Moreover, by \eqref{eq21}, we have
	\begin{equation*}
		\begin{aligned}
			\mathcal{L}_t\left(x(t),y_t\right) - \mathcal{L}_t\left(x_t,y(t)\right) =& f(x(t))-f(x_t)-\langle \nabla f(x_t), x(t)-x_t \rangle \\
			&+g(y(t))-g(y_t)-\langle \nabla g(y_t),y(t)-y_t \rangle \\
			&+\frac{\epsilon(t)}{2}\left(\|x(t)-x_t\|^2 + \|y(t)-y_t\|^2\right).
		\end{aligned}
	\end{equation*}
	By the convexity of $f$ and $g$,
	\begin{align*}
		f(x(t))-f(x_t)-\langle \nabla f(x_t), x(t)-x_t \rangle \geq 0,
		\\[5pt]
		g(y(t))-g(y_t)-\langle \nabla g(y_t),y(t)-y_t \rangle \geq 0.
	\end{align*}
	Consequently,
	\begin{equation}\label{lyapunov'}
		\widehat{\mathcal{E}}(t) \geq \frac{\beta(t)\epsilon(t)}{2}\left(\|x(t)-x_t\|^2+\|y(t)-y_t\|^2\right).
	\end{equation}
	Combining \eqref{lyapunov} with \eqref{lyapunov'} yields
	\begin{equation}\label{ready}
		\|x(t)-x_t\|^2 + \|y(t)-y_t\|^2 \leq \frac{2t_0^2\widehat{\mathcal{E}}(t_0)}{t^2\beta(t)\epsilon(t)} + \frac{2\int_{t_0}^{t}\tau^2\mathcal{R}(\tau)\mathrm{d}\tau}{t^2\beta(t)\epsilon(t)}.
	\end{equation}
	Since $\lim\limits_{t\to+\infty}t^2\beta(t)\epsilon(t)=+\infty$, condition \eqref{N2} shows that the distance between $\left(x(t),y(t)\right)$ and $\left(x_t,y_t\right)$ tends to zero, namely,
	\begin{equation}\label{jixian1}
		\lim\limits_{t\to+\infty} \|\left(x(t),y(t)\right)-\left(x_t,y_t\right)\| = 0.
	\end{equation}
	The triangle inequality further gives
	\begin{equation*}
		\|\left(x(t),y(t)\right)-\left(\hat{x}^*,\hat{y}^*\right)\| \leq \|\left(x(t),y(t)\right)-\left(x_t,y_t\right)\| + \|\left(x_t,y_t\right)-\left(\hat{x}^*,\hat{y}^*\right)\|.
	\end{equation*}
	Lemma \ref{strongshoulian} and \eqref{jixian1}, together with the triangle inequality, yield
	\begin{equation*}
		\lim\limits_{t\to+\infty} \|\left(x(t),y(t)\right)-\left(\hat{x}^*,\hat{y}^*\right)\| = 0.
	\end{equation*}
	
	This completes the proof of Theorem \ref{theorem3}.
\end{proof}

Condition \eqref{N2} is formulated at the level of the time-dependent regularized saddle path rather than directly on the original trajectory. It measures the accumulated effect of the motion of the regularized saddle point relative to the tracking gain generated by the Tikhonov term. Therefore, \eqref{N2} provides a sufficient condition for controlling the deviation between the generated trajectory and the time-dependent regularized saddle path. Compared with trajectory localization conditions, this formulation emphasizes the role of the regularized saddle path in the strong convergence mechanism.

\begin{proposition}\label{prop:N1prime}
	Suppose that $t_0 \geq 1$, that Assumption \ref{assumption} and conditions \eqref{tiaojian2'}, \eqref{tiaojian3}, and \eqref{tiaojian4'} hold, and that $\lim\limits_{t\to+\infty}t^2\beta(t)\epsilon(t)=+\infty$. Define
	\begin{equation*}
	\Psi(t):=\left(\frac{|\dot{\epsilon}(t)|}{\epsilon(t)}\right)^2\left(\frac{H^2(t)}{l(t)}+\frac{1}{t^2\Gamma(t)}\right),
	\end{equation*}
	where $H(t):=\beta(t)+\beta(t)\epsilon(t)+\alpha(t)-\frac{1}{\theta t}$.
	If
	\begin{equation}\label{N2'}
	\lim_{t\to+\infty}
	\frac{
		\displaystyle\int_{t_0}^{t}\tau^2\Psi(\tau)\,\mathrm{d}\tau
	}{
		t^2\beta(t)\epsilon(t)
	}
	=0,
	\tag{$\mathcal{N}_2$}
	\end{equation}
	then condition \eqref{N2} is satisfied. Consequently, the trajectory generated by the dynamical system \eqref{eq15} converges strongly to the minimum-norm saddle point $(\hat{x}^*,\hat{y}^*)$.
\end{proposition}

\begin{proof}
	By Lemmas \ref{strongshoulian} and \ref{cunzai}, for almost every
	$t\geq t_0$,
	\[
	\|x_t\|^2+\|y_t\|^2 \leq \|(\hat{x}^*,\hat{y}^*)\|^2,
	\qquad
	\|\dot{x}_t\|^2+\|\dot{y}_t\|^2 \leq \left( \frac{|\dot{\epsilon}(t)|}{\epsilon(t)} \right)^2 \|(\hat{x}^*,\hat{y}^*)\|^2.
	\]
	Using the definitions of $P_x(t)$ and $P_y(t)$, together with the Cauchy--Schwarz inequality, we obtain
	\[
		\|P_x(t)\|^2+\|P_y(t)\|^2 \leq  D_3\|(\hat{x}^*,\hat{y}^*)\|^2 \left( \frac{|\dot{\epsilon}(t)|}{\epsilon(t)} \right)^2 
		\left( \beta^2(t)+\beta^2(t)\epsilon^2(t)+\left(\alpha(t)-\frac{1}{\theta t}\right)^2 \right),
	\]
	for some constant $D_3>0$ that depends only on $\theta$ and $\|K\|$ and is independent of $t$.
	
	Since \eqref{tiaojian2'} implies $\alpha(t)-\frac{1}{\theta t}>0$, the definition of $H(t)$ yields
	\[
	\beta^2(t)+\beta^2(t)\epsilon^2(t)
	+\left(
	\alpha(t)-\frac{1}{\theta t}
	\right)^2
	\leq H^2(t).
	\]
	Hence, by the definition of $\mathcal{R}(t)$,
	\[
	\mathcal{R}(t)\leq D_3\|(\hat{x}^*,\hat{y}^*)\|^2\left(\frac{|\dot{\epsilon}(t)|}{\epsilon(t)}\right)^2
	\left(\frac{H^2(t)}{l(t)}+\frac{1}{t^2\Gamma(t)}\right)=D_3\|(\hat{x}^*,\hat{y}^*)\|^2\Psi(t).
	\]
	Therefore,
	\[
	0\leq
	\frac{\displaystyle\int_{t_0}^{t}\tau^2\mathcal{R}(\tau)\,\mathrm{d}\tau}{t^2\beta(t)\epsilon(t)}
	\leq
	D_3\|(\hat{x}^*,\hat{y}^*)\|^2
	\frac{\displaystyle\int_{t_0}^{t}\tau^2\Psi(\tau)\,\mathrm{d}\tau}{t^2\beta(t)\epsilon(t)}.
	\]
	Thus, condition \eqref{N2'} implies \eqref{N2}, and the conclusion follows from Theorem \ref{theorem3}.
\end{proof}

\begin{remark}\label{remark:N1relation}
	Condition \eqref{N2} is a path-level condition involving the evolution of the regularized saddle path. It directly quantifies the influence of the moving regularization target on the tracking error. Condition \eqref{N2'} is a parameter-level sufficient condition obtained by estimating the quantities appearing in \eqref{N2}. Although \eqref{N2'} may be more restrictive than \eqref{N2}, it has the advantage of being verifiable directly from the prescribed coefficient functions and the extrapolation parameter $\theta$, without requiring an explicit expression for the regularized saddle path.
	
	To show that condition \eqref{N2'} is nonvacuous, consider
	\[
	\alpha(t)=\frac{\alpha+\delta\sin(\log t)}{t},
	\qquad
	\beta(t)=t^\beta,
	\qquad
	\epsilon(t)=\frac{c}{(\log t)^r},
	\]
	where $\beta\geq0$, $c>0$, $\delta>0$, $0<r<1$, and $\alpha-\delta>\beta+3$. Choose $\frac{1}{\alpha-\delta-1}<\theta\leq\frac{1}{\beta+2}$. Then, for $t_0>1$ sufficiently large, the above parameters satisfy conditions \eqref{tiaojian2'}, \eqref{tiaojian3}, \eqref{tiaojian4'}, and \eqref{N2'}, and $\lim\limits_{t\to+\infty}t^2\beta(t)\epsilon(t)=+\infty$. Therefore, Proposition \ref{prop:N1prime} and Theorem~\ref{theorem3} imply that every global solution $(x(t),y(t))$ of the dynamical system \eqref{eq15} converges strongly to the minimum-norm saddle point $(\hat{x}^{*},\hat{y}^{*})=\operatorname{Proj}_{\Omega}0$, namely,
	\[
	\lim_{t\to+\infty}\left\|(x(t),y(t))-(\hat{x}^{*},\hat{y}^{*})\right\|=0.
	\]
\end{remark}

We finally consider the case in which the minimum-norm saddle point $(\hat{x}^*,\hat{y}^*)$ is the origin. Then the regularized saddle path is identically zero and the moving-target forcing terms vanish. As a result, the path-motion condition \eqref{N2} is no longer needed, and the non-strict conditions \eqref{tiaojian2}--\eqref{tiaojian4} suffice for strong convergence.

    \begin{corollary}\label{corollary1}
    	Let $(x(t),y(t))$ be a global solution of the dynamical system \eqref{eq15}, and suppose that the minimum-norm saddle point $\left(\hat{x}^*,\hat{y}^*\right)$ is the origin. If Assumption \ref{assumption} and conditions \eqref{tiaojian2}--\eqref{tiaojian4} hold and if $\lim\limits_{t\to+\infty}t^2\beta(t)\epsilon(t)=+\infty$, then the trajectory converges strongly to the minimum-norm saddle point $(\hat{x}^*,\hat{y}^*)$, that is,
    	\begin{equation*}
    		\lim_{t\to+\infty} \left\|(x(t),y(t))-(\hat{x}^*,\hat{y}^*)\right\| = 0.
    	\end{equation*}
    \end{corollary}
    
    \begin{proof}
    	Lemma \ref{strongshoulian} gives $\|\left(x_t,y_t\right)\| \leq \|\left(\hat{x}^*,\hat{y}^*\right)\| = 0$. Hence, $x_t=y_t=0$ for all $t \geq t_0$.
    	
    	In this case, the same energy function $\widehat{\mathcal{E}}(t)$ remains applicable. The last three terms in \eqref{strong} vanish, so \eqref{strong} reduces to
    	\begin{equation*}
    		\begin{aligned}
    			\frac{2}{t}\widehat{\mathcal{E}}(t)+\dot{\widehat{\mathcal{E}}}(t)\leq & \left(\dot{\beta}(t)+\left(\frac{2}{t}-\frac{1}{\theta t}\right)\beta(t)\right) \left(\mathcal{L}_t\left(x(t),y_t\right) - \mathcal{L}_t\left(x_t,y(t)\right)\right) \\
    			&-\frac{l(t)}{2}\left(\|x(t)-x_t\|^2+\|y(t)-y_t\|^2\right) \\
    			&-\Gamma(t)\left(\|\dot{x}(t)\|^2+\|\dot{y}(t)\|^2\right).
    		\end{aligned}
    	\end{equation*}
    	Using \eqref{eq19'} and conditions \eqref{tiaojian2}--\eqref{tiaojian4}, the above differential inequality yields $\widehat{\mathcal{E}}(t) \leq \frac{t_0^2\widehat{\mathcal{E}}(t_0)}{t^2}$. Inequality \eqref{lyapunov'} then yields
    	\begin{equation*}
    		\|x(t)-x_t\|^2 + \|y(t)-y_t\|^2 \leq \frac{2t_0^2\widehat{\mathcal{E}}(t_0)}{t^2\beta(t)\epsilon(t)}.
    	\end{equation*}
    	Together with $\lim\limits_{t\to+\infty}t^2\beta(t)\epsilon(t)=+\infty$, this shows that the trajectory converges strongly to the minimum-norm saddle point, namely,
    	\begin{equation*}
    		\left(x(t),y(t)\right) \to \left(\hat{x}^*,\hat{y}^*\right).
    	\end{equation*}
    	
    	This completes the proof of Corollary \ref{corollary1}.
    \end{proof}

\section{Numerical Experiments}\label{sec4}
In this section, we present three numerical examples to illustrate the finite-time behavior and numerical performance of the proposed system for selected parameter settings. The numerical solutions are computed in MATLAB R2023b using the adaptive Runge--Kutta method implemented in \texttt{ode45}. All simulations are performed on a personal computer equipped with a 1.60 GHz Intel Core i5-10210U processor and 16 GB of RAM.

Before presenting the numerical experiments, we clarify the terminology used below. The slowly damped Tikhonov-regularized inertial primal-dual dynamical system introduced in \cite{ref44} is referred to as (SDTRD). Consistent with the terminology used in the preceding sections, the dynamical systems introduced in \cite{ref40} and \cite{ref42} are referred to as (APDD) and (MPDD), respectively, while the dynamical system considered in this paper is denoted by \eqref{eq15}.

\begin{example}\label{example1}
	\textbf{A Min-Max Optimization Problem}
\end{example}

	Let $x := (x_1,x_2)^{T} \in \mathbb{R}^{2}$ and $y := (y_1,y_2)^{T} \in \mathbb{R}^{2}$. Consider the following convex-concave min-max problem \cite{ref45}:
	\begin{equation}\label{e1}
		\min_{x \in \mathbb{R}^2}\max_{y \in \mathbb{R}^2} \left(mx_1+nx_2\right)^2+\langle Kx,y \rangle-\left(jy_1+ky_2\right)^2,
	\end{equation}
	where $m,n,j,k \in \mathbb{R}\setminus\{0\}$ and
	$K=\begin{pmatrix}
		mj & nj\\
		mk & nk
	\end{pmatrix}$.
	A direct calculation shows that the saddle point set of the problem \eqref{e1} is
	$\left\{ (x,y) \in \mathbb{R}^2 \times \mathbb{R}^2 \mid mx_1 + nx_2 = 0 \text{ and } jy_1 + ky_2 = 0 \right\}$.
	Hence, $\left( \hat{x}^*,\hat{y}^* \right)=\left(0,0,0,0\right)^T$ is the minimum-norm saddle point of the problem \eqref{e1}, with the corresponding saddle value equal to $0$.
	
	For the numerical experiments, we take $(x^*,y^*)=(\hat{x}^*,\hat{y}^*)=(0,0,0,0)^T$. We set $m=1,n=10,j=15,k=1,\alpha=17,\beta=1,\alpha(t)=\frac{\alpha}{t},\beta(t)=t^\beta,\theta=\frac{2}{\alpha-1}$. The initial time is $t_0=1$, and the dynamical systems are numerically integrated over the time interval $[1,50]$ with the following initial conditions:
	\begin{equation}\label{chushitiaojian}
		x(t_0)=
		\begin{pmatrix}
			1\\
			1.5
		\end{pmatrix},
		\quad
		\dot{x}(t_0)=
		\begin{pmatrix}
			1\\
			1
		\end{pmatrix},
		\quad
		y(t_0)=
		\begin{pmatrix}
			1\\
			1.5
		\end{pmatrix},
		\quad
		\dot{y}(t_0)=
		\begin{pmatrix}
			1\\
			1
		\end{pmatrix}.
	\end{equation}
	
	In the first experiment, we set the Tikhonov regularization parameter to $\epsilon(t)=\frac{7}{t^{2}}$ and examine the behavior of the trajectory relative to the minimum-norm saddle point $(\hat{x}^{*},\hat{y}^{*})=(0,0,0,0)^{T}$. Using the initial conditions given in \eqref{chushitiaojian}, Fig.~\ref{Fig2}(a) shows the components of the trajectory generated by \eqref{eq15} with $\epsilon(t)=\frac{7}{t^{2}}$, while Fig.~\ref{Fig2}(b) shows the corresponding trajectory components for the dynamical system without Tikhonov regularization, that is, with $\epsilon(t)\equiv 0$.
	\begin{figure}[htbp]
		\centering
		\subfigure[With Tikhonov regularization]{%
			\includegraphics[width=0.475\textwidth]{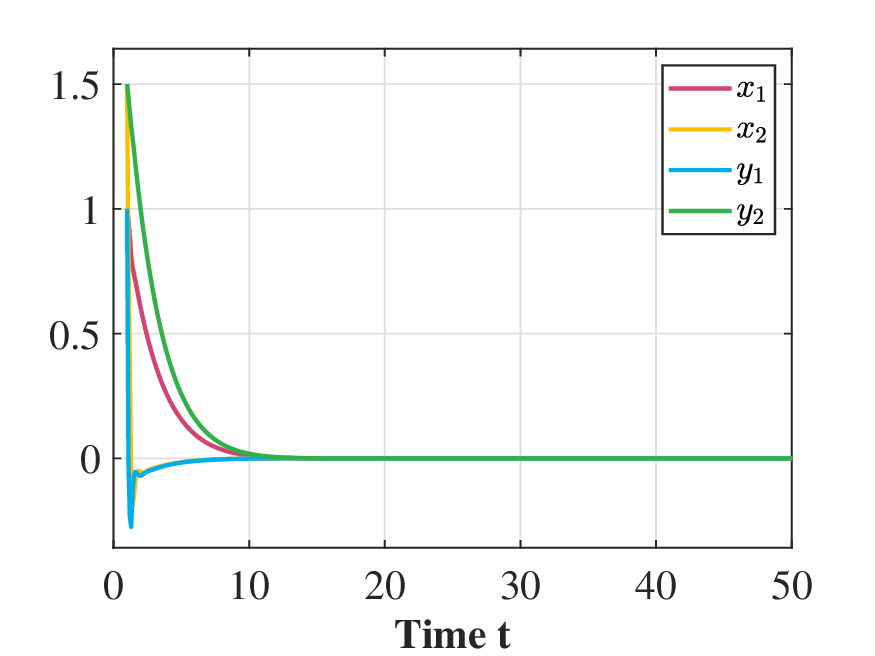} }
		\subfigure[Without Tikhonov regularization]{%
			\includegraphics[width=0.475\textwidth]{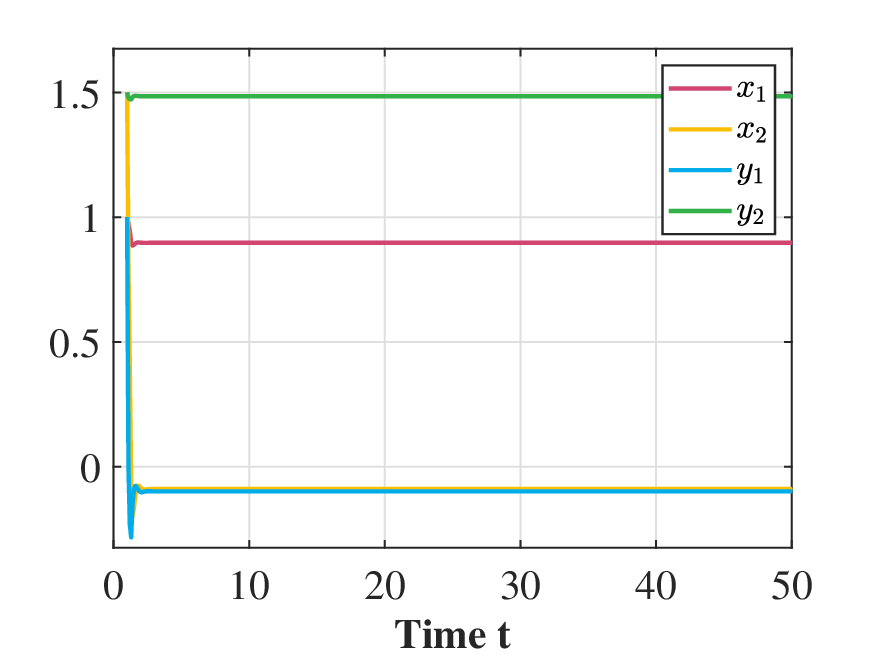} }
		\caption{Effect of Tikhonov regularization on the convergence toward the minimum‑norm saddle point}
		\label{Fig2} 
	\end{figure}
	
	Fig.~\ref{Fig2} illustrates that, with Tikhonov regularization, the trajectory $\left(x(t),y(t)\right)$ approaches the minimum-norm saddle point $\left(\hat{x}^{*},\hat{y}^{*}\right)=\left(0,0,0,0\right)^{T}$. In contrast, the unregularized trajectory fails to exhibit comparable convergence toward the minimum-norm saddle point.

	In the second numerical experiment, we examine how different decay exponents of the Tikhonov regularization parameter affect the primal-dual gap $\mathcal{L}(x(t),y^{*})-\mathcal{L}(x^{*},y(t))$, the trajectory error $\|x(t)-x^{*}\|+\|y(t)-y^{*}\|$, and the velocity measure $\|\dot{x}(t)\|+\|\dot{y}(t)\|$ along the trajectory generated by \eqref{eq15}. Specifically, we set $\epsilon(t)=\frac{7}{t^r}$ with $1<r<3$ and consider $r\in\{1.2,1.6,2.0,2.4,2.8\}$.
	
    \begin{figure}[htbp]
    	\centering
    	
    	\subfigure[Decay of the primal-dual gap]{%
    		\includegraphics[width=0.475\textwidth]
    		{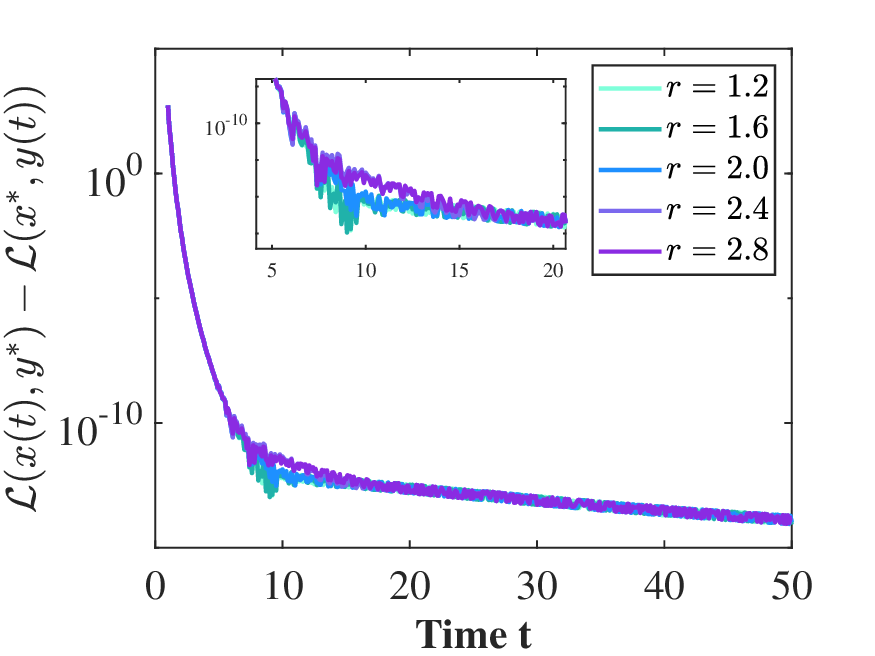}
    		\label{Fig1a}
    	}
    	
    	\vspace{0.6em}
    	
    	\subfigure[Decay of the trajectory error]{%
    		\includegraphics[width=0.475\textwidth]
    		{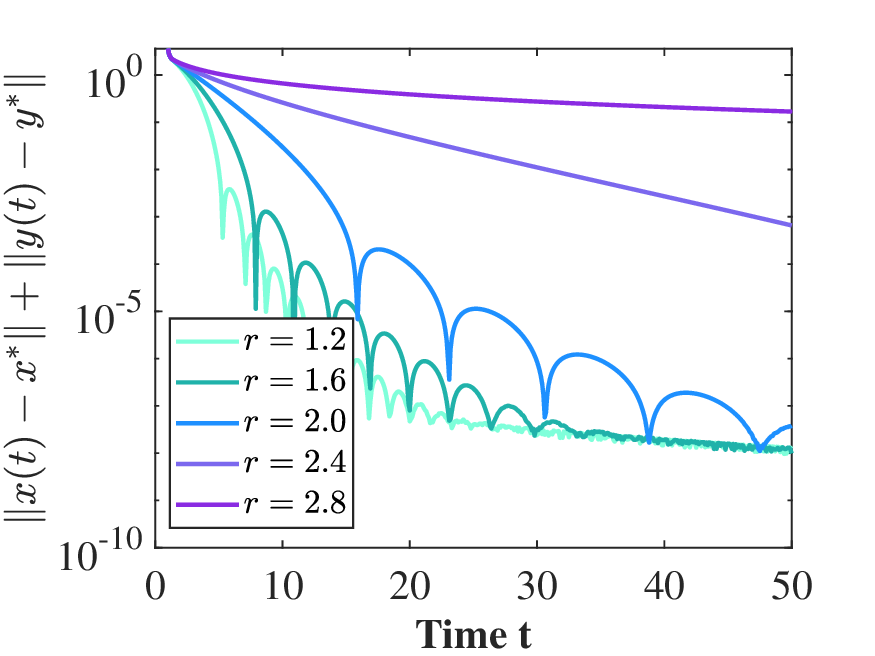}
    		\label{Fig1b}
    	}
    	\subfigure[Decay of the velocity measure]{%
    		\includegraphics[width=0.475\textwidth]
    		{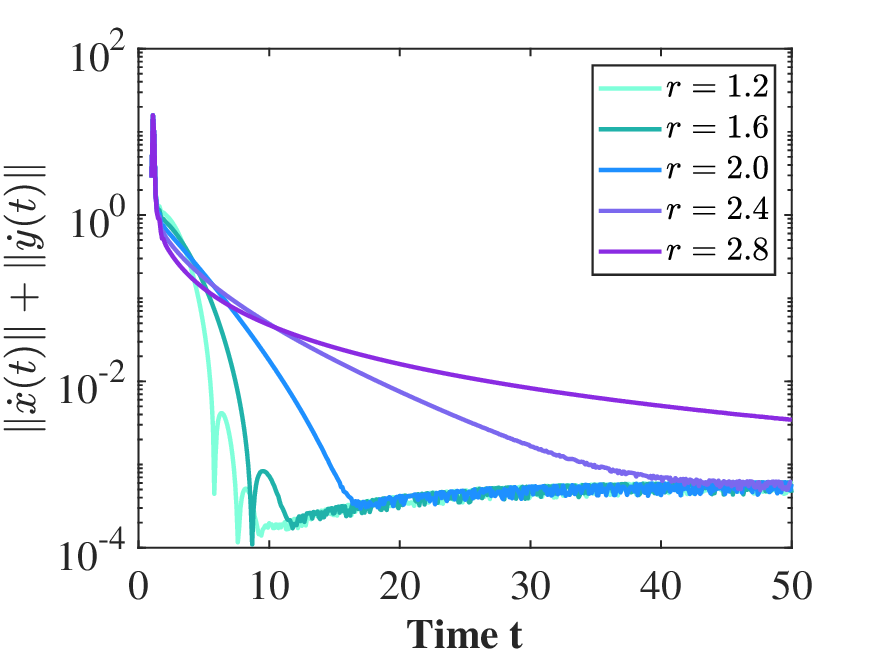}
    		\label{Fig1c}
    	}
    	
    	\caption{Convergence behavior of \eqref{eq15} under different decay exponents $r$ for the Tikhonov regularization parameter}
    	\label{Fig1}
    \end{figure}
	
	Fig.~\ref{Fig1} illustrates that the primal–dual gap $\mathcal{L}\left(x(t),y^*\right)-\mathcal{L}\left(x^*,y(t)\right)$ is relatively insensitive to variations in the decay exponent $r$ of the Tikhonov regularization parameter. Moreover, as $r$ decreases within the interval $1<r<3$, both the trajectory error $\left\|x(t)-x^*\right\|+\left\|y(t)-y^*\right\|$ and the velocity measure $\left\|\dot{x}(t)\right\|+\left\|\dot{y}(t)\right\|$ tend to decay more rapidly and attain smaller values over the considered time interval.

    In the third numerical experiment, we compare the decay behavior of the primal-dual gap for (APDD), (SDTRD), and \eqref{eq15}. For each dynamical system, numerical simulations are performed under two different parameter settings, as specified in Table~\ref{table1}.
    
	\begin{table}[htbp]
		
		\centering
		
		\caption{Parameter settings used in the third numerical experiment for Example~\ref{example1}}
		
		\label{table1}
		
		\renewcommand{\arraystretch}{2}
		
		\begin{tabular*}{\linewidth}{@{\extracolsep{\fill}} llccccc @{}}
			
			\toprule
			
			Dynamical system & Setting & $\alpha(t)$ & $\beta(t)$ & $\theta(t)$ & $\epsilon(t)$ & Reference \\
			
			\midrule
			
			\multirow{2}{*}[-0.5ex]{APDD}
			& Case 1 & $\dfrac{2}{t}$ & $1$ & $\dfrac{3t}{4}$ & $0$ & \cite[the case $0<\alpha<3$]{ref40} \\
			
			& Case 2 & $\dfrac{17}{t}$ & $1$ & $\dfrac{t}{2}$ & $0$ & \cite[the case $\alpha>3$]{ref40} \\
			
			\cmidrule(lr){1-7}
			
			\multirow{2}{*}[-0.5ex]{SDTRD}
			& Case 1 & $\dfrac{17}{t^{0.8}}$ & $t$ & $\dfrac{t^{0.8}}{16}$ & $\dfrac{7}{t^{3}}$ & \cite[Theorem~3.1]{ref44} \\
			
			& Case 2 & $\dfrac{17}{t^{0.8}}$ & $t$ & $\dfrac{t^{0.8}}{16}$ & $\dfrac{7}{t^{2}}$ & \cite[Theorem~3.2]{ref44} \\
			
			\cmidrule(lr){1-7}
			
			\multirow{2}{*}[-0.5ex]{GTRD}
			& Case 1 & $\dfrac{17}{t}$ & $t$ & $\dfrac{t}{8}$ & $\dfrac{7}{t^2}$ & \multirow{2}{*}[-0.5ex]{Theorem \ref{theorem2}} \\
			
			& Case 2 & $\dfrac{17}{t} + \dfrac{1}{t^2}$ & $t$ & $\dfrac{t}{8}$ & $\dfrac{7}{t^2}$ & \\
			
			\bottomrule
			
		\end{tabular*}
		
	\end{table}

	\begin{figure}[htbp]
		\centering
		\includegraphics[width=0.75\textwidth]{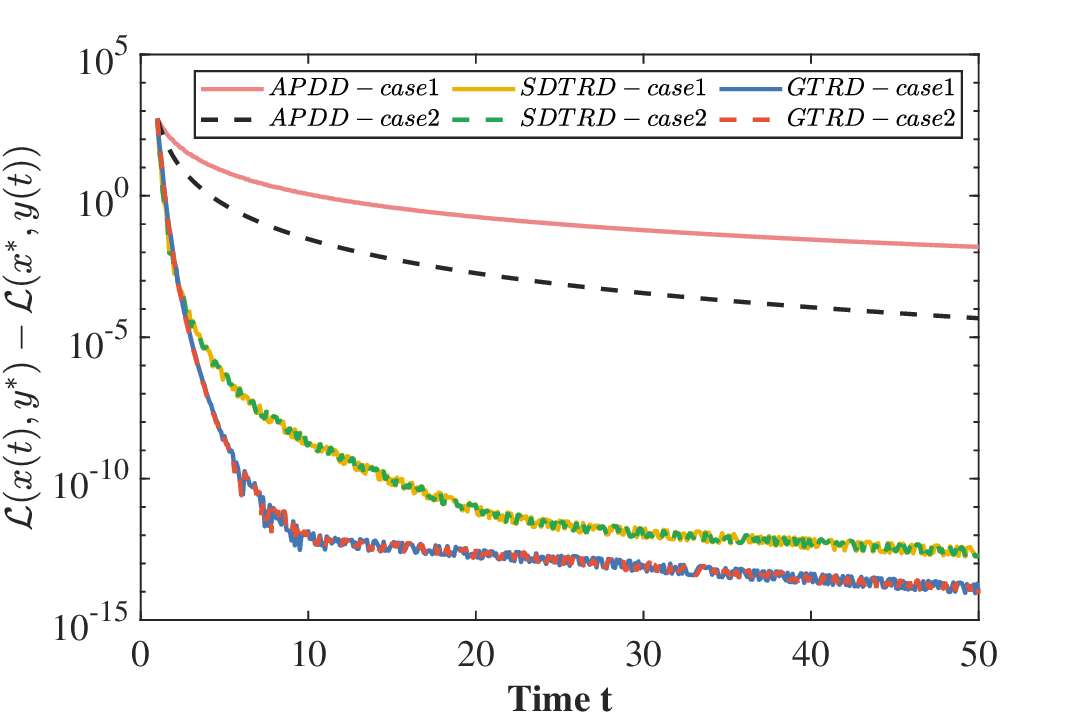}
		\caption{Comparison of primal‑dual gap decay curves for (APDD), (SDTRD), and \eqref{eq15}}
		\label{Fig3} 
	\end{figure}
	
    Fig.~\ref{Fig3} indicates that, under the tested parameter settings, \eqref{eq15} generally attains smaller primal-dual gap values than (APDD) and (SDTRD) over the considered simulation interval.

\begin{remark}\label{shuzhishiyan1}
	In the first numerical experiment, the case $\epsilon(t) \equiv 0$ corresponds to (MPDD) with $\theta(t) = \theta t$, as introduced in \cite{ref42}. The numerical results in Fig.~\ref{Fig2} indicate that the trajectory generated by \eqref{eq15} converges toward the minimum‑norm saddle point, whereas the trajectory generated by (MPDD) does not exhibit comparable convergence toward this point over the considered time interval.
\end{remark}

\begin{example}
	\textbf{A Rank-Deficient Quadratic Saddle Point Problem}
\end{example}

Let $x:=(x_1,x_2)^T\in\mathbb{R}^2$ and $y:=(y_1,y_2)^T\in\mathbb{R}^2$. Consider the following convex-concave min-max problem:
\begin{equation}
	\min_{x\in\mathbb{R}^2}\max_{y\in\mathbb{R}^2}
	\frac{\mu}{2}
	\left(x_1-1-\frac{2}{\mu}\right)^2
	+\langle Kx,y\rangle
	-\frac{\mu}{2}
	\left(y_1-2+\frac{1}{\mu}\right)^2,
	\label{eq:strong-example}
\end{equation}
where $\mu>0$ and $K=
\begin{pmatrix}
	1 & 0\\
	0 & 0
\end{pmatrix}$.
A direct calculation shows that the saddle point set of \eqref{eq:strong-example} is $\Omega=\left\{ (x,y) \in \mathbb{R}^2 \times \mathbb{R}^2 \mid x_1=1 \text{ and } y_1=2 \right\}$. Hence, the unique minimum-norm saddle point is not the origin and is given by $(\hat{x}^{*},\hat{y}^{*})=(1,0,2,0)^T$.

For the numerical experiment, we set $\mu=1000$, $\alpha(t)=\frac{20+0.5\sin(\log t)}{t}$, $\beta(t)=1$, $\theta=0.1$, $\epsilon(t)=\frac{0.03}{\sqrt{\log t}}$. These parameter choices satisfy the sufficient condition \eqref{N2'}, as illustrated by the example following Remark \ref{remark:N1relation}. The initial time is $t_0=5$, and the dynamical system is numerically integrated over the time interval $[5,200]$ with the following initial conditions:
\begin{equation*}
	x(t_0)=
	\begin{pmatrix}
		2.5\\
		1.5
	\end{pmatrix},
	\qquad
	\dot{x}(t_0)=
	\begin{pmatrix}
		0.2\\
		-0.3
	\end{pmatrix},
	\qquad
	y(t_0)=
	\begin{pmatrix}
		-0.5\\
		-1
	\end{pmatrix},
	\qquad
	\dot{y}(t_0)=
	\begin{pmatrix}
		-0.2\\
		0.4
	\end{pmatrix}.
	\label{eq:strong-example-initial}
\end{equation*}

Fig.~\ref{Fig4} compares the trajectories generated by \eqref{eq15} with and without Tikhonov regularization. The regularized trajectory converges strongly toward the nonzero minimum-norm saddle point $(\hat{x}^{*},\hat{y}^{*})=(1,0,2,0)^T$, whereas the unregularized trajectory does not exhibit comparable convergence over the considered time interval.
\begin{figure}[htbp]
	\centering
	\subfigure[With Tikhonov regularization]{%
		\includegraphics[width=0.475\textwidth]{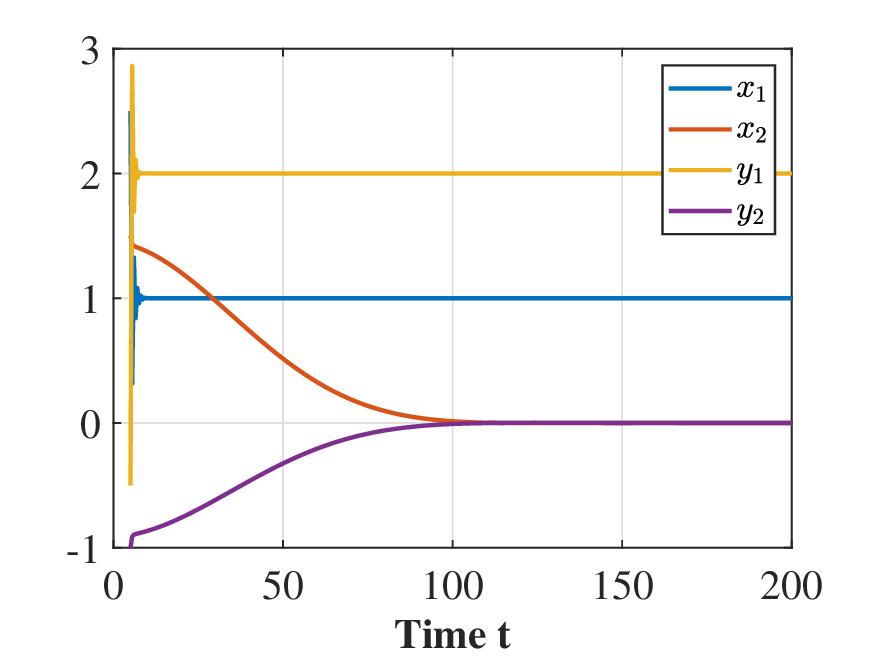} }
	\subfigure[Without Tikhonov regularization]{%
		\includegraphics[width=0.475\textwidth]{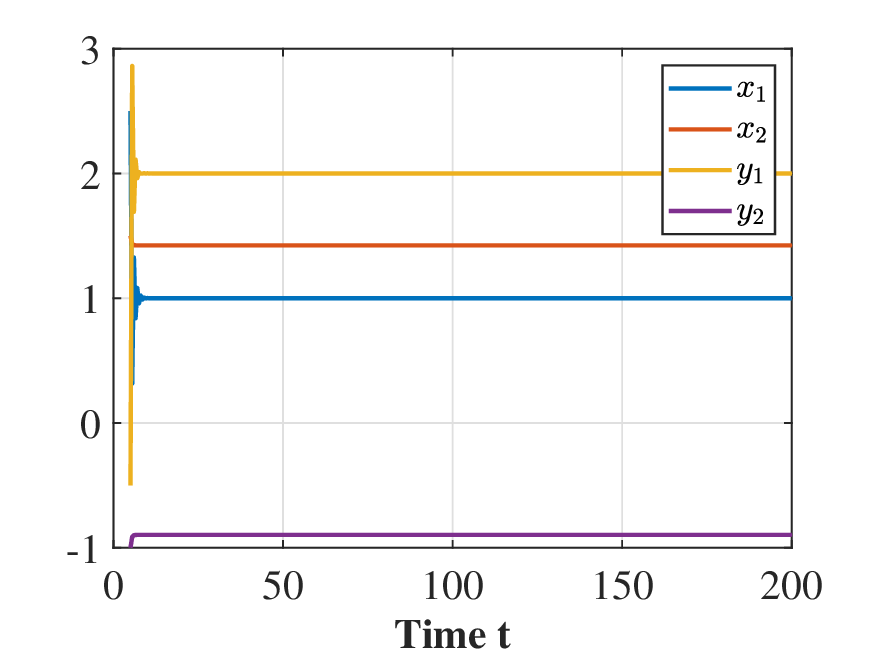} }
	\caption{Effect of Tikhonov regularization on the convergence toward the nonzero minimum‑norm saddle point}
	\label{Fig4} 
\end{figure}

\begin{example}\label{example2}
    \textbf{A Linear Regression Problem}
\end{example}
	Consider the following linear regression problem:
	\begin{equation*}\label{l2}
		\min_{x \in \mathbb{R}^n} \Phi(x) = \frac{1}{2} \left\| Kx - b \right\|^2 + \omega \mathcal{R}^a(x),
	\end{equation*}
	where $K \in \mathbb{R}^{m \times n}$, $b \in \mathbb{R}^m$, $\omega > 0$, $a>0$, and
	\begin{equation*}
		\mathcal{R}^a(x) = \sum_{i=1}^n \frac{1}{a}\left(\log(1 + \exp(a x_i)) + \log(1 + \exp(-a x_i))\right)
	\end{equation*}
	is a smoothed $L_1$-regularization term (see \cite{shuzhi}). The corresponding convex-concave saddle point formulation can be written as
	\begin{equation}\label{l2'}
		\min_{x \in \mathbb{R}^n} \max_{y \in \mathbb{R}^m} \omega \mathcal{R}^a(x) + \langle Kx, y \rangle - \left( \frac{1}{2} \| y \|^2 + \langle b, y \rangle \right).
	\end{equation}
	
	We first generate an auxiliary matrix $K_{0}\in\mathbb{R}^{m\times n}$ with independent standard normal entries and compute its singular value decomposition $K_{0}=U\Sigma V^{T}$. For a prescribed condition number $\kappa_{K}$, we set $\sigma_{\min}=0.1$ and $\sigma_{\max}=\kappa_{K}\sigma_{\min}$. The intermediate singular values are sampled independently from a log-uniform distribution on $[\sigma_{\min},\sigma_{\max}]$, while the smallest and largest singular values are fixed at the two endpoints. We then reconstruct the matrix as $K=U\widetilde{\Sigma}V^{T}$, thereby ensuring that $\kappa(K)=\kappa_{K}$ up to numerical precision.
	
	We set $\omega=0.1$ and $a=100$, and generate $b\sim\mathcal{N}(0,I_m)$. With the random seed fixed at $42$, the numerical experiments are performed for $\kappa(K)\in\{35,200\}$ with $n=200$ and $m=100$. The initial conditions are given by
	\[
	x(t_0)=\dot{x}(t_0)=\mathbf{1}_{n},
	\qquad
	y(t_0)=\dot{y}(t_0)=\mathbf{1}_{m},
	\]
	where $\mathbf{1}_{n}\in\mathbb{R}^{n}$ and $\mathbf{1}_{m}\in\mathbb{R}^{m}$ denote the all-ones vectors.

	In the first numerical experiment, conducted over the time interval $[1,20]$, we set $\alpha=7$, $\beta=1$, $\alpha(t)=\frac{\alpha}{t}$, $\beta(t)=t^\beta$, $\theta=\frac{1}{5}$, and $\epsilon(t)=\frac{c}{t^r}$. To examine the influence of Tikhonov regularization on the primal objective value $\Phi(x(t))$ along the trajectories associated with \eqref{l2'}, we consider $c=1$ and $r\in\{1.5,2.0,2.5\}$. The resulting objective values are compared with those obtained from the corresponding dynamical system without Tikhonov regularization, that is, with $c=0$.
	
	\begin{figure}[htbp]
		\centering
		\subfigure[$\kappa(K) = 35$]{%
			\includegraphics[width=0.475\textwidth]{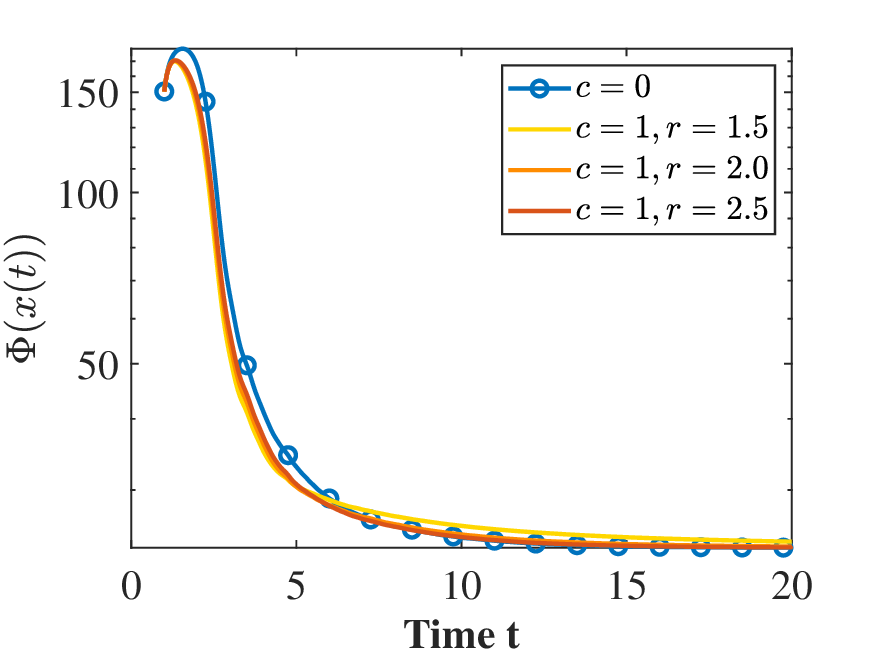} }
		\subfigure[$\kappa(K) = 200$]{%
			\includegraphics[width=0.475\textwidth]{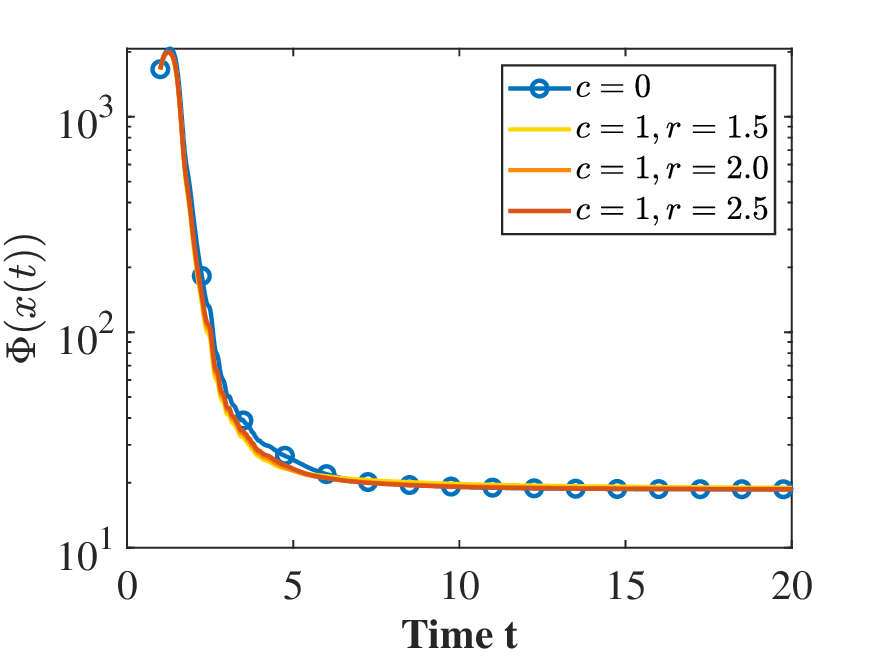} }
		\caption{Convergence of $\Phi(x(t))$ for $n = 200$ and $m = 100$}
		\label{Fig4.1} 
	\end{figure}
	
	Fig.~\ref{Fig4.1} illustrates the following numerical observations:
	\begin{itemize}
		\item[(1)] For both condition-number settings, the objective value $\Phi(x(t))$ along the trajectory generated by the dynamical system \eqref{eq15} with Tikhonov regularization generally exhibits faster decay than that along the trajectory generated by the corresponding unregularized system over the considered time interval.
		\item[(2)] For the tested values $r\in\{1.5,2.0,2.5\}$, the decay exponent $r$ exerts only a limited influence on the objective value $\Phi(x(t))$, since the corresponding curves remain close over the considered time interval.
	\end{itemize}
    
    In the second numerical experiment, we compare the evolution of the objective value $\Phi\left(x(t)\right)$ for (APDD), (SDTRD), and \eqref{eq15} over the time interval $[1,50]$. Each dynamical system is simulated under two different parameter settings, as specified in Table~\ref{table2}.
    
    \begin{table}[htbp]
    	
    	\centering
    	
    	\caption{Parameter settings used in the second numerical experiment for Example~\ref{example2}}
    	
    	\label{table2}
    	
    	\renewcommand{\arraystretch}{2}
    	
    	\begin{tabular*}{\linewidth}{@{\extracolsep{\fill}} llccccc @{}}
    		
    		\toprule
    		
    		Dynamical system & Setting & $\alpha(t)$ & $\beta(t)$ & $\theta(t)$ & $\epsilon(t)$ & Reference \\
    		
    		\midrule
    		
    		\multirow{2}{*}[-0.5ex]{APDD}
    		& Case 1 & $\dfrac{2}{t}$ & $1$ & $\dfrac{3t}{4}$ & $0$ & \cite[the case $0<\alpha<3$]{ref40} \\
    		
    		& Case 2 & $\dfrac{7}{t}$ & $1$ & $\dfrac{t}{2}$ & $0$ & \cite[the case $\alpha>3$]{ref40} \\
    		
    		\cmidrule(lr){1-7}
    		
    		\multirow{2}{*}[-0.5ex]{SDTRD}
    		& Case 1 & $\dfrac{6}{t^{0.2}}$ & $t^{0.1}$ & $\dfrac{t^{0.2}}{5}$ & $\dfrac{10}{t^{2}}$ & \multirow{2}{*}[-0.5ex]{\cite[Example 5.2]{ref44}} \\
    		
    		& Case 2 & $\dfrac{6}{t^{0.6}}$ & $t^{0.3}$ & $\dfrac{t^{0.6}}{5}$ & $\dfrac{10}{t^{2}}$ & \\
    		
    		\cmidrule(lr){1-7}
    		
    		\multirow{2}{*}[-0.5ex]{GTRD}
    		& Case 1 & $\dfrac{7}{t}$ & $t$ & $\dfrac{t}{5}$ & $\dfrac{10}{t^2}$ & \multirow{2}{*}[-0.5ex]{Theorem \ref{theorem2}} \\
    		
    		& Case 2 & $\dfrac{7}{t} + \dfrac{1}{t^2}$ & $t$ & $\dfrac{t}{5}$ & $\dfrac{10}{t^2}$ & \\
    		
    		\bottomrule
    		
    	\end{tabular*}
    	
    \end{table}

    \begin{figure}[htbp]
    	\centering
    	\subfigure[$\kappa(K) = 35$]{%
    		\includegraphics[width=0.475\textwidth]{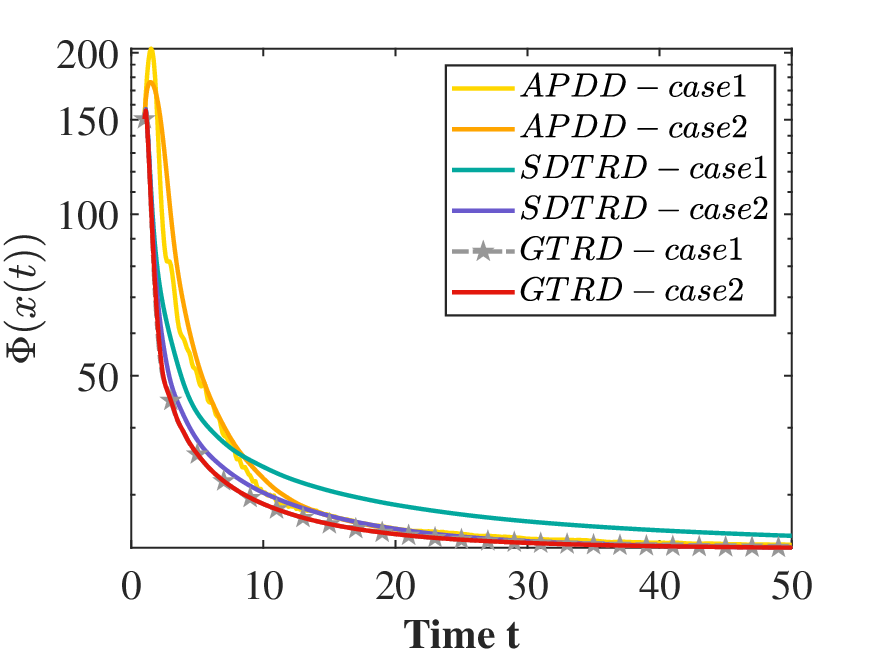} }
    	\quad
    	\subfigure[$\kappa(K) = 200$]{%
    		\includegraphics[width=0.475\textwidth]{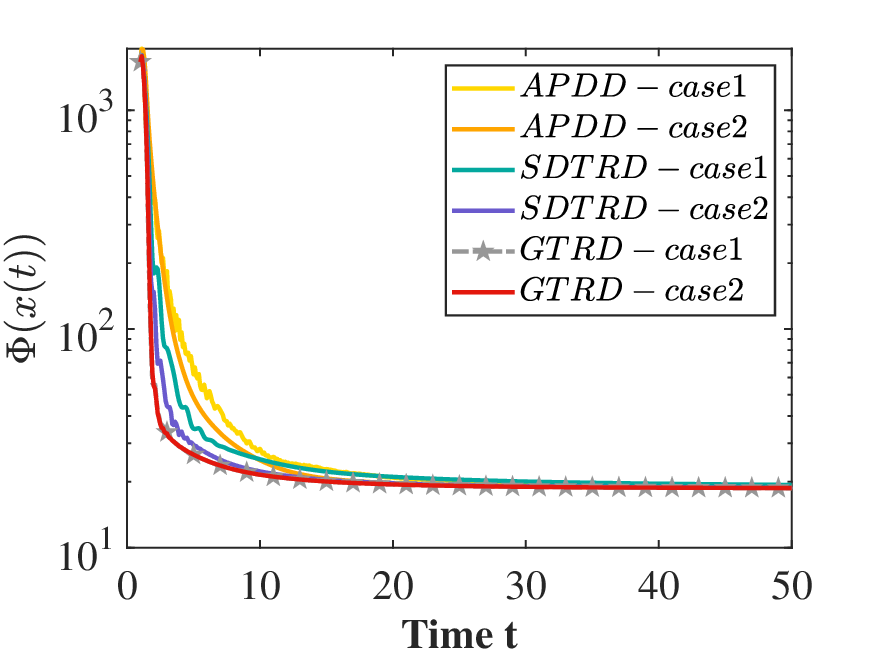} }
    	\caption{Convergence of $\Phi\left(x(t)\right)$ for $n = 200$ and $m = 100$}
    	\label{Fig5.1} 
    \end{figure}
    
    The numerical results in Fig.~\ref{Fig5.1} indicate the following behavior:
    \begin{itemize}
    	\item[(1)] For the tested parameter settings, the objective value along the trajectory generated by \eqref{eq15} generally decreases more rapidly during the initial transient stage and attains smaller values than the corresponding objective values for (APDD) and (SDTRD) over the considered time interval.
    	\item[(2)] These results further suggest that the relative advantage of \eqref{eq15} persists across the two tested condition-number settings.
    \end{itemize}


\section{Conclusion}\label{sec5}

In this paper, we propose a Tikhonov-regularized second-order primal-dual dynamical system with time-dependent damping and time scaling for solving convex-concave bilinear saddle point problems. We develop a Lyapunov-based coefficient-compatible framework to characterize, through differential balance conditions, the joint admissible regime of the viscous damping coefficient $\alpha(t)$, the time-scaling coefficient $\beta(t)$, the extrapolation parameter $\theta$, and the Tikhonov regularization parameter $\epsilon(t)$, thereby extending the analysis beyond prescribed damping profiles.

Within this framework, we investigate two regimes corresponding to different decay rates of the Tikhonov regularization parameter. Under relatively fast decay, we establish the convergence rate $\mathcal{L}(x(t),y^*)-\mathcal{L}(x^*,y(t))=\mathcal{O}\left(\frac{1}{t^2\beta(t)}\right)$ for the primal-dual gap, together with estimates for the trajectory velocities, gradient differences, and several weighted integrals. Under slower decay, we further obtain $\mathcal{L}(x(t),y^*)-\mathcal{L}(x^*,y(t))=o\left(\frac{1}{\beta(t)}\right)$, provided that $\beta(t)\to+\infty$.

For the minimum-norm saddle point selection problem, we construct a moving-saddle Lyapunov framework centered at the instantaneous Tikhonov-regularized saddle point. Under an appropriate path-motion condition, we prove that the entire trajectory converges strongly to the minimum-norm saddle point of the original problem, without directly imposing an eventual geometric localization condition on the original trajectory. Furthermore, we derive a parameter-based sufficient condition that makes the required path-motion condition explicitly verifiable and present a special corollary for the case in which the minimum-norm saddle point is the origin.

Finally, numerical experiments illustrate the theoretical results and the finite-time behavior of the proposed dynamical system \eqref{eq15} under representative parameter settings.

\section*{Acknowledgements}

This research was supported by the National Natural Science Foundation of China (62371094).

\backmatter

\begin{appendices}
	
\end{appendices}

\bibliography{sn-bibliography}

\end{document}